\documentclass[a4paper,10pt, reqno]{amsart}

\usepackage[english]{babel}
\usepackage[utf8]{inputenc}

\usepackage{amssymb}
\usepackage{amsmath, braket, mathtools}
\usepackage{amsthm}
\usepackage{bbm}
\usepackage{dutchcal}
\usepackage{caption}
\usepackage{dsfont}
\usepackage[shortlabels]{enumitem}
\usepackage{mathrsfs}

\usepackage{tikz, xcolor, pgfplots}
\definecolor{dgreen}{rgb}{0,0.5,0}
\pgfplotsset{compat=newest}
\usepackage{marginnote}

\usepackage[colorlinks=true]{hyperref}
\hypersetup{urlcolor=blue, citecolor=red}

\newcommand{\bbG}{\mathbb{G}}
\newcommand{\bbH}{\mathbb{H}}

\newcommand{\bbS}{\mathbb{S}}

\newcommand{\calA}{\mathcal{A}}
\newcommand{\calB}{\mathcal{B}}
\newcommand{\calC}{\mathcal{C}}
\newcommand{\calD}{\mathcal{D}}

\newcommand{\calH}{\mathcal{H}}

\newcommand{\calL}{\mathcal{L}}

\newcommand{\calS}{\mathcal{S}}
\newcommand{\calT}{\mathcal{T}}

\newcommand{\calu}{\mathcal{u}}
\newcommand{\calv}{\mathcal{v}}
\newcommand{\calf}{\mathcal{f}}
\newcommand{\calg}{\mathcal{g}}

\newcommand{\fra}{\mathfrak{a}}

\newcommand{\frh}{\mathfrak{h}}
\renewcommand{\frq}{\mathfrak{q}}

\newcommand{\Dm}{D_{\mathrm{max}}}
\newcommand{\conormal}{\partial^L_\nu}
\newcommand{\weakstar}{\rightharpoonup^*}

\newcommand{\R}{\mathds{R}}
\newcommand{\C}{\mathds{C}}
\newcommand{\N}{\mathds{N}}

\newcommand{\suchthat}{~|~} 
\DeclareMathOperator{\one}{\mathbbm{1}} 
\DeclareMathOperator{\dist}{dist} 
\newcommand{\argument}{\mathord{\,\cdot\,}} 
\newcommand{\dd}{\;\mathrm{d}} 
\DeclareMathOperator{\linSpan}{span} 
\DeclareMathOperator{\Div}{div}

\DeclareMathOperator{\trace}{tr}

\let\Re\undefined
\DeclareMathOperator{\Re}{Re}
\let\Im\undefined
\DeclareMathOperator{\Im}{Im}

\newcommand{\spb}{\mathrm{s}} 

\newcommand{\eps}{\varepsilon}
\newcommand{\ltwoog}{\ensuremath{L^2(\Omega)\times L^2(\Gamma)}}
\newcommand{\linftyog}{\ensuremath{L^\infty(\Omega)\times L^\infty(\Gamma)}}

\newcommand{\oug}{{\Omega\sqcup\Gamma}}

\theoremstyle{definition}
\newtheorem{defn}{Definition}[section]
\newtheorem{rem}[defn]{Remark}

\newtheorem{example}[defn]{Example}
\newtheorem{hyp}[defn]{Hypothesis}

\theoremstyle{plain}
\newtheorem{prop}[defn]{Proposition}
\newtheorem{lem}[defn]{Lemma}
\newtheorem{thm}[defn]{Theorem}
\newtheorem{cor}[defn]{Corollary}

\numberwithin{equation}{section}

\title{Title}

\author{Markus Kunze}
\address[M. Kunze]{University of Konstanz, Fachbereich Mathematik und Statistik, Fach 193, 78357, Konstanz, Germany}
\email{markus.kunze@uni-konstanz.de}

\author{Jonathan Mui}
\address[J. Mui]{University of Wuppertal, School of Mathematics und Natural Sciences, Gaußstr.\ 20, 42119 Wuppertal, Germany}
\email{jomui@uni-wuppertal.de}

\author{David Plo\ss}
\address[D. Plo\ss]{Karlsruhe Institute of Technology, Department of Mathematics, Englerstraße 2, 76131 Karlsruhe, Germany}
\email{david.ploss@kit.edu}

\title[Non-local Wentzell--Robin boundary conditions]{Elliptic operators with \\non-local Wentzell--Robin boundary conditions}

\date{\today}

\begin{document}

\begin{abstract}
    In this article, we study strictly elliptic, second-order differential operators on a bounded Lipschitz domain in $\R^d$
    subject to certain non-local Wentzell--Robin boundary conditions. We prove that such operators generate strongly continuous semigroups
    on $L^2$-spaces and on spaces of continuous functions. We also provide a characterization of positivity and (sub-)Mar\-kovianity
    of these semigroups. Moreover, based on spectral analysis of these operators, we discuss further properties of the semigroup such 
    as asymptotic behavior and, in the case of a non-positive semigroup, the weaker notion of eventual positivity of the semigroup.
\end{abstract}

\maketitle

\section{Introduction}

Let $\Omega\subset \R^d$ be a bounded domain with Lipschitz boundary $\Gamma = \partial\Omega$.
In this article, we study strictly elliptic second-order differential operators, such as the Laplacian $\Delta$ (which we will 
consider for the rest of this introduction), subject to the non-local Wentzell--Robin boundary condition
\begin{equation}
    \label{eq.nonlocalbc}
    \partial_\nu u(z) = \Delta u(z) + \int_\Omega b_{21}(z,x) u(x)\dd \lambda(x) + \int_\Gamma b_{22}(z,w) u(w)\dd\sigma(w)
\end{equation}
for $z\in \Gamma$. 
Here $\lambda$ denotes the Lebesgue measure on $\Omega$, $\sigma$ denotes the $(d-1)$-dimensional Hausdorff measure on $\Gamma$, and 
$b_{21}\in L^\infty(\Gamma\times\Omega)$ and $b_{22} \in L^\infty(\Gamma\times\Gamma)$ are suitable integral kernels that account for
the non-locality of the boundary condition. If $b_{21}=0$ and $b_{22}=0$, we obtain local Wentzell--Robin boundary conditions. 

In contrast to the classical Dirichlet, Neumann, and Robin boundary conditions, the boundary trace of the differential operator occurs in the local Wentzell--Robin boundary condition. The presence of this term accounts for the fact that there might be some energy concentrated
on the boundary $\Gamma$. We refer to \cite{g06} for a derivation and physical interpretation of these boundary conditions. 
Second-order differential operators with local boundary conditions of Wentzell--Robin type have been widely studied in the literature, see
for example
\cite{fggr02, ampr03, e05, ef05, Nittka, warma, fggo16, be19}.\smallskip

The motivation to consider non-local terms in boundary conditions stems from probability theory and goes back to the pioneering work 
of Feller \cite{feller52, feller54}, who studied one-dimensional diffusion processes. A probabilistic interpretation of boundary conditions
may be found in \cite{im63}, see also the recent article \cite{bobrowski24}. Roughly speaking, the kernels $b_{21}$ and $b_{22}$ are
responsible for restarting the associated stochastic process when it reaches the point $z\in \Gamma$.

However, in addition to this probabilistic motivation, there are also real-world models in which differential operators subject
to non-local boundary conditions naturally appear. Some examples include thermoelasticity~\cite{day82}, a model of Bose condensation~\cite{s89}, or a model of a thermostat~\cite{gm97}.\smallskip

An important question is whether the differential operator subject to the boundary condition \eqref{eq.nonlocalbc} generates
a strongly continuous semigroup, i.e.\ whether the associated Cauchy problem is well-posed. Additional properties of the semigroup
such as positivity, (sub-)Markovianity, and asymptotic behavior are also of interest. For the semigroup to be the transition semigroup 
of a stochastic process, it is of course essential that it is Markovian. However, in some real-world applications the semigroup fails to be Markovian or even positive. In this case, it is rather natural to ask whether the semigroup is \emph{eventually positive}
in some sense. A theory of eventually positive semigroups has been established in \cite{DGK2, DGK1, DG18}. As it turns out, the question of whether a given semigroup is eventually positive can be deduced from \emph{spectral information} about its generator. In \cite{GM24}, the authors
investigated the question of whether a symmetric, strictly elliptic second-order differential operator subject to non-local Robin boundary 
conditions (i.e.\ \eqref{eq.nonlocalbc} without the Laplacian term and with $b_{21}=0$) generates an eventually positive semigroup.\medskip

In this article, we carry out a systematic investigation of strictly elliptic operators of second order, subject to general non-local
Wentzell--Robin boundary conditions~--- see Hypothesis \ref{h.main} for our standing assumptions on the coefficients of the differential operator and the boundary conditions. There are two primary objectives.

Firstly, we want to establish that our operator generates a strongly continuous semigroup and characterize when this semigroup
is positive and/or (sub-)~Markovian. Our main results in this direction are Theorems \ref{t.generate}, \ref{t.characterization}, and \ref{t.sgonc}.
These results complement results in the literature concerning non-local Dirichlet boundary conditions (see \cite{gs01, bap07, bap09, AKK16})
and non-local Robin boundary conditions (see \cite{sku89, tai16, AKK18}). In the case of positive semigroups, we also obtain a complete characterization of the asymptotic behavior of the semigroup in Theorem \ref{t.positiveasymptotic}.

Secondly, we want to identify situations in which the associated semigroup is eventually positive but not positive. As we remarked
above, this requires us to study certain spectral properties of the generator. Unsurprisingly, it is very hard to discuss these questions
in full generality, as spectral properties may depend intricately on the coefficients. Therefore, we focus more on examples for this second point.
In Theorem \ref{t.unif-eventual-pos}, we identify a concrete subclass of operators for which the semigroup is eventually positive. 
However, it may happen that even eventual positivity fails. In Section \ref{sect.perturbation}, we discuss a specific one-dimensional 
example depending on a real parameter $\tau$, where eventual positivity fails for certain choices of $\tau$. As a matter of fact,
Theorem \ref{t.ev.pos} shows that  different spectral phenomena may be responsible for the failure of eventual positivity.

\subsection*{Acknowledgments}
The authors thank Wolfgang Arendt and Jochen Gl\"{u}ck for many insightful discussions and comments. 
The second author acknowledges the financial support of DFG Project~515394002. 
The third author acknowledges the financial support of DFG Project~258734477~(SFB 1173).

\section{Second order elliptic operators}

Throughout this article $\Omega\subset \R^d$ denotes a bounded domain with Lipschitz boundary $\Gamma$. We denote Lebesgue measure on $\Omega$
by $\lambda$ and surface measure on $\Gamma$ (i.e.\ $(d-1)$-dimensional Hausdorff measure) by $\sigma$. For $p\in [1,\infty]$
the corresponding complex 
$L^p$-spaces are denoted by $L^p(\Omega)$ and $L^p(\Gamma)$ respectively and the corresponding norms are $\|\cdot\|_{\Omega, p}$
and $\|\cdot\|_{\Gamma, p}$. Similarly, the scalar products on $L^2(\Omega)$ and $L^2(\Gamma)$ are denoted by $\langle\cdot,\cdot\rangle_\Omega$ 
and $\langle\cdot, \cdot\rangle_\Gamma$ respectively. 
The classical Sobolev space of square integrable functions on $\Omega$ with weak derivative in $L^2(\Omega)$
is denoted by $H^1(\Omega)$.

In this section, we define a uniformly elliptic second-order differential operator on $L^2(\Omega)$ that will play a central role throughout.
The following are our standing assumptions on the coefficients.
\begin{hyp}
    \label{h.prelim}
    $\Omega\subset \R^d$ is a bounded domain with Lipschitz boundary $\Gamma$. We are given a function 
    $A=(a_{ij}) \in L^\infty(\Omega; \R^{d\times d})$ and $b=(b_j), c=(c_j) \in L^\infty(\Omega; \R^d)$.
    The matrix $A=(a_{ij})$ is assumed to be symmetric (i.e. $a_{ij}= a_{ji}$ for all $i,j=1, \ldots, d$) 
    and uniformly elliptic in the sense that there exists a constant $\eta>0$ such that 
    \[
    \sum_{i,j=1}^d a_{ij}(x)\xi_i\bar{\xi}_j \geq \eta |\xi|^2
    \]
    for all $\xi\in \C^d$ and almost all $x\in \Omega$.
\end{hyp}

We now define the distributional operator $\mathfrak{L}: H^1(\Omega) \to \calD(\Omega)'$ by setting
\begin{equation}
\label{eq.ldist}
\langle \mathfrak{L}u, \varphi\rangle \coloneqq \sum_{i,j=1}^d \int_\Omega a_{ij}D_iu\overline{D_j\varphi}\, \dd\lambda 
+\sum_{j=1}^d \int_\Omega b_j (D_ju)\overline{\varphi} + c_j u\overline{D_j\varphi}\dd\lambda 
\end{equation}
for all $\varphi\in C_c^\infty(\Omega)$. Here, $\calD(\Omega)'$ refers to the space of all anti-linear distributions.
The maximal $L^2$-realization of $\mathfrak{L}$ is denoted by $L$, i.e.
\begin{equation}
    \label{eq.l2}
    \begin{aligned}
    \Dm(L) &\coloneqq \big\{ u\in H^1(\Omega) \suchthat \exists\, f\in L^2(\Omega) \mbox{ with }\\
    & \qquad\qquad \langle \mathfrak{L}u, \varphi\rangle = \int_\Omega 
    f\overline{\varphi}\dd\lambda\,\, \forall\, \varphi\in C_c^\infty(\Omega) \big\} \\ 
    Lu &= f.
    \end{aligned}
\end{equation}

\emph{Sectorial forms} play an important role in this article. Form methods provide useful tools to establish well-posedness of certain
Cauchy problems on a Hilbert space $H$. Indeed, if the operator $A$ is associated to a closed, densely defined and sectorial form, then $-A$ generates
an analytic, strongly continuous semigroup on $H$, see Section 1.4 of \cite{Ouhabaz}. For our purposes, 
we introduce the sectorial (actually, symmetric) form 
$\frq : H^1(\Omega)\times H^1(\Omega) \to \C$ by setting
\begin{equation}
    \label{eq.qform}
    \frq[u,v]\coloneqq \sum_{i,j=1}^d \int_\Omega a_{ij}D_iu\overline{D_jv}\, \dd\lambda 
+\sum_{j=1}^d \int_\Omega b_j (D_ju)\overline{v} + c_j u\overline{D_jv}\dd\lambda.
\end{equation}
Noting that $\frq[u,\varphi]=\langle Lu, \varphi\rangle$ for all $u\in \Dm(L)$ and $\varphi\in C_c^\infty(\Omega)$ we see
that the associated operator is a suitable realization of $L$ on $L^2(\Omega)$. We use a larger class of test functions to define the
\emph{weak conormal derivative}.

\begin{defn}
    Let $u\in \Dm(L)$. We say that $u$ \emph{has a weak conormal derivative in $L^2(\Gamma)$} if there exists a function $g\in L^2(\Gamma)$
    such that 
    \begin{equation}\label{eq.conormal}
    \frq[u,v]-\langle Lu, v\rangle_\Omega = \int_\Gamma g\;\overline{\trace v}\dd\sigma 
    \end{equation}
    for all $v\in H^1(\Omega)$. In this case, we set $\conormal u \coloneqq g$. Occasionally, we abbreviate the statement that
    $u$ has a weak conormal derivative in $L^2(\Gamma)$ by merely writing $\conormal u \in L^2(\Gamma)$.
\end{defn}

Under our assumptions on the coefficients, given $u\in H^1(\Omega)$, there is at most one function $g$ satisfying \eqref{eq.conormal}. Thus,
the conormal derivative is unique whenever it exists. Moreover, it depends on the operator $L$ only through the coefficients $A$ and $c$. 
If these coefficients are smooth enough to have a trace on the boundary, it can be shown that
\begin{equation}
\label{eq.formulanormal}
\conormal u = \sum_{j=1}^d\Big(\sum_{i=1}^d \trace a_{ij}D_iu +\trace c_j u\Big)\nu_j,
\end{equation}
where $\nu = (\nu_1, \ldots, \nu_d)$ denotes the unit outer normal of $\Omega$, which exists $\sigma$-almost everywhere on $\Gamma$. For
a proof of these facts and further information, we refer to \cite[Section 8.1]{A15}.

It is immediate from the definition of the conormal, that the domain of the operator associated to the form $\frq$ is given by
$\{ u\in \Dm(L) \suchthat \conormal u = 0\}$. More generally, we have the following:

\begin{lem}
    \label{l.weaksol}
    Let $f\in L^2(\Omega)$ and $g\in L^2(\Gamma)$. If $u\in H^1(\Omega)$ satisfies
    \[
    \frq[u,v] = \int_\Omega f\overline{v}\dd \lambda +\int_\Gamma g\;\overline{\trace v}\dd \sigma, 
    \] 
    for all $v\in H^1(\Omega)$, 
    then $u\in \Dm(L)$, $Lu=f$ and $\conormal u = g$. In this case, we say that $u$ is a \emph{weak solution} of the Neumann boundary value problem
    \begin{equation}
        \label{eq.neumannproblem}
        \begin{cases}
            Lu & = f \qquad\text{in }\Omega, \\
            \conormal u & = g \qquad\text{on }\Gamma.
        \end{cases}
    \end{equation}
\end{lem}

\begin{proof}
    That $u\in \Dm(L)$ with $Lu=f$ follows by considering $v\in C_c^\infty(\Omega)$. But then
    \[
    \frq[u,v] -\langle Lu, v\rangle_\Omega = \int_\Gamma g\;\overline{\trace v}\dd\sigma
    \]
    for all $v\in H^1(\Omega)$ and $\conormal u = g$ follows from the definition of the weak conormal derivative.
\end{proof}
Similarly we may also consider $\frq^\lambda[u,v]:=\frq[u,v]+\lambda \langle u,v \rangle_\Omega$ for $\lambda \in \R$ to define a weak solution $u \in H^1(\Omega)$ of the shifted Neumann problem
 \begin{equation}
        \label{eq.neumannproblem.lambda}
        \left\{ \begin{aligned}
            (\lambda+L)u & = f \qquad\text{in }\Omega, \\
            \conormal u & = g \qquad\text{on }\Gamma,
        \end{aligned}\right.
    \end{equation} 
which is uniquely solvable for large enough $\lambda$ by \cite[Proposition 3.7]{Nittka}.
We next recall a regularity result from \cite{Nittka}.
\begin{lem}
    \label{l.nittkahoelder}
    Let $d\geq 2$ and  $\eps>0$ be given.  Then there exist constants $\alpha \in (0,1)$ and $C>0$ such that the following holds:
    \begin{enumerate}[\upshape (i)]
        \item Let  $\lambda$ large enough, $f\in L^{d-1+\eps}(\Omega)$ and $g\in L^{d-1+\eps}(\Gamma)$. Then the Neumann problem \eqref{eq.neumannproblem.lambda} has a unique weak solution $u$. This solution $u$ belongs to $C^{\alpha}(\Omega)$ and
    \[
    \label{eq.nittkahoelder}
    \|u\|_{C^{\alpha}} \leq C\big(\|f\|_{\Omega, d-1+\eps} + \|g\|_{\Gamma, d-1+\eps}\big).
    \]
    \item If $u \in H^1(\Omega)$ is a weak solution of \eqref{eq.neumannproblem} and additionally satisfies $u \in L^{d-1+\eps}(\Omega)$, then
    \[
    \label{eq.nittkahoelder2}
    \|u\|_{C^{\alpha}} \leq C\big(\|u\|_{\Omega,d-1+\eps}+\|f\|_{\Omega, d-1+\eps} + \|g\|_{\Gamma, d-1+\eps}\big).
    \]
    \end{enumerate} 
\end{lem}

\begin{proof}
For (i), see \cite[Lemma 3.10]{Nittka}. For (ii) note that if $u$ solves \eqref{eq.neumannproblem}, then $u$ also solves 
    \begin{equation}
        \left\{ \begin{aligned}
            (\lambda+L) u &= f+\lambda u \qquad &\text{in }\Omega, \\
            \conormal u &= g \qquad &\text{on }\Gamma
        \end{aligned}\right.
    \end{equation}
    for every $\lambda$. If $\lambda$ is large enough, part (i) yields
    \[
    \|u\|_{C^{\alpha}} \leq C\big(\|f+\lambda u\|_{\Omega, d-1+\eps} + \|g\|_{\Gamma, d-1+\eps}\big). \qedhere
    \]
\end{proof}

\begin{cor}
    \label{c.imorove}
    Let $d\geq 3$, $\eps>0$ be given and define $\psi : [2,\infty) \to [2, \infty]$ by 
    \[
     \psi(p) \coloneqq 
     \begin{cases} \frac{d- 3 + \eps}{d-1+\eps-p} p, & \mbox{ for } 2\leq p < d-1+\eps,\\
     \infty, & \mbox{ for } p\geq d-1+\eps.
     \end{cases}
    \]
    Then, if $f\in L^p(\Omega)$, $g\in L^p(\Gamma)$, and $u$ is a weak solution of~\eqref{eq.neumannproblem}, then
    $u\in L^{\psi(p)}(\Omega)$ and $\trace u \in L^{\psi(p)}(\Gamma)$. Moreover, there is a constant
    $C$ independent of $f$ and $g$ such that
    \[
    \|u\|_{\Omega, \psi(p)} +\|\trace u\|_{\Gamma, \psi(p)} \leq C\big(\|u\|_{\Omega, p} +\|f\|_{\Omega, p} + \|g\|_{\Gamma, p}\big).
    \]
\end{cor}

\begin{proof}
We fix $\lambda \in \R$ large enough so that \eqref{eq.neumannproblem.lambda} is uniquely solvable. 
Note that for  $u\in C^{\alpha}(\Omega)$ we have $u\in L^\infty(\Omega)$ and $\trace u \in L^\infty(\Gamma)$. 
Thus,  Lemma \ref{l.nittkahoelder}(i) implies that there exists a constant $C$ such that 
    \begin{equation}
        \label{eq.infinty}
        \|\tilde u\|_{\Omega, \infty} + \|\trace \tilde u\|_{\Gamma, \infty} \leq C\big(\|\tilde f\|_{\Omega, d-1+\eps} + \|\tilde g\|_{\Gamma, d-1+\eps}\big)
    \end{equation}
    whenever $\tilde f\in L^{d-1+\eps}(\Omega)$, $\tilde g\in L^{d-1+\eps}(\Gamma)$, and $\tilde u$ is the unique weak solution of the Neumann problem~\eqref{eq.neumannproblem.lambda} with right-hand sides $\tilde f$ and $\tilde g$.

    On the other hand, \cite[Lemma 3.7 and 3.8]{Nittka} yield that there is a constant $C$ such that for $\tilde f\in L^2(\Omega)$ and
    $\tilde g\in L^2(\Gamma)$ and the weak solution $\tilde u$ of~\eqref{eq.neumannproblem.lambda} it holds that
    \begin{equation}
        \label{eq.2}
        \|\tilde u\|_{\Omega, 2} +\|\trace \tilde u\|_{\Gamma, 2} \leq C \big(\|\tilde f\|_{\Omega, 2} + \|\tilde g\|_{\Gamma, 2}\big).
    \end{equation}

    We may now use an interpolation argument similar to \cite[Lemma 3.11]{Nittka}. We put $X_0\coloneqq L^2(\Omega)\times L^2(\Gamma)$,
    $X_1\coloneqq L^{d-1+\eps}(\Omega)\times L^{d-1+\eps}(\Gamma)$, $Y_0 \coloneqq  L^2(\Omega)\times L^2(\Gamma)$
    and $Y_1\coloneqq L^\infty(\Omega)\times L^\infty(\Gamma)$. Consider the unique solution operator to problem \eqref{eq.neumannproblem.lambda}
    \[
    R_0 : X_0 \to Y_0, \quad (\tilde f,\tilde g) \mapsto (\tilde u, \trace \tilde u)
    \]
    which is continuous by \eqref{eq.2}. By \eqref{eq.infinty}, its restriction $R_1\coloneqq R_0|_{X_1}$ is also continuous
    from $X_1$ to $Y_1$. Using complex interpolation, it follows that $R_\theta \coloneqq R_0|_{[X_0:X_1]_\theta}$ is
    continuous from the interpolation space $[X_0:X_1]_\theta$ to $[Y_0:Y_1]_\theta$. Using that complex interpolation
    is compatible with Cartesian products (see \cite[\S 1.9]{Tri95}) and the standard identification of interpolation of $L^p$-spaces,
    it follows that
    \[
    \|\tilde u\|_{\Omega, \psi(p)} + \|\trace \tilde u\|_{\Gamma, \psi(p)} \leq C \big(\|\tilde f\|_{\Omega, p} + \|\tilde g\|_{\Gamma, p}\big)
    \]
    where $\frac{1}{\psi(p)} = \frac{(1-\theta)}{2}$ for the unique solution $\theta$ of $\frac{1}{p}=\frac{1-\theta}{2} + \frac{\theta}{d-1+\eps}$.
    
    Now let $u$ be a (not necessarily unique) weak solution of \eqref{eq.neumannproblem}. As in the proof of Lemma \ref{l.nittkahoelder} $u$ will also satisfy
    \[ \frq^\lambda[u,v]=\int_\Omega (f+\lambda u)\; \bar v \; \mathrm{d}x + \int_\Gamma g\; \overline{\trace v} \; \mathrm{d}\sigma\]
and thus coincide with the unique solution of 
 \begin{equation*}
        \left\{ \begin{aligned}
            (\lambda+L)\tilde u & = f+\lambda u \qquad &\text{in }\Omega, \\
            \conormal \tilde u & = g \qquad &\text{on }\Gamma,
        \end{aligned}\right.
    \end{equation*}
whence 
$\|u\|_{\Omega,p}+\|\trace u\|_{\Gamma,p}=\|\tilde u\|_{\Omega,p}+\| \trace \tilde u\|_{\Omega,p}\leq \|f+\lambda u\|_{\Omega,p}+\|g\|_{\Gamma,p}$, from which the claim follows.
\end{proof}

\section{The sectorial form}

We are now ready to introduce the Wentzell--Robin boundary conditions. We will work on the  Hilbert space $\calH = L^2(\Omega)\times L^2(\Gamma)$ and write $\calu = (u_1, u_2)\in \calH$. The scalar product on $\calH$ is defined by
\[
\langle \calu, \calv\rangle_\calH \coloneqq \langle u_1, v_1\rangle_\Omega + \langle u_2, v_2\rangle_\Gamma.
\]
Occasionally, we identify $\calH$ with $L^2(\oug)$, where we endow the disjoint
union $\oug$ with the product measure $\lambda\otimes\sigma$.

We point out that $\calH$ consists of complex-valued functions. The \emph{real part} of $\calH$ is the subspace
\begin{equation*}
    \calH_\R \coloneqq \{ \calu\in\calH \suchthat u_1\in L^2(\Omega;\R), u_2\in L^2(\Gamma;\R)\}
\end{equation*}
consisting of real-valued functions. A linear operator $\calS$ on $\calH$ is called \emph{real} if $\calS(\calH_\R)\subset \calH_\R$.\smallskip

The following are our standing assumptions:

\begin{hyp}\label{h.main}
    Assume Hypothesis \ref{h.prelim} and let the
    operators $B_{11}\in \calL(L^2(\Omega))$, $B_{22}\in \calL(L^2(\Gamma))$, $
    B_{12} \in \calL(L^2(\Gamma), L^2(\Omega))$ and $B_{21} \in \calL(L^2(\Omega), L^2(\Gamma))$ be \emph{real}
    in the sense that they map real-valued functions to real-valued functions. We define the bounded linear operator
    $\calB \in \calL(\calH)$ by setting
    \[
    \calB = \begin{pmatrix}
        B_{11} & B_{12}\\
        B_{21} & B_{22}
    \end{pmatrix}.
    \]
\end{hyp}

\begin{example}
    \label{ex.integraloperator}
    Motivated by the probabilistic interpretation from the introduction, an important example for the 
    operators $B_{kl}$ is given by \emph{integral operators}. Let us briefly recall the definition and some important
    properties of integral operators. Let $(S_j, \Sigma_j, \mu_j)$ be finite measure spaces for $j=1,2$. An operator
    $K\in \calL(L^2(S_1), L^2(S_2))$ is called \emph{integral operator} if there exists a product measurable map
    $k: S_1\times S_2 \to \C$ -- the \emph{kernel} of the integral operator -- such that
    \[
    [Kf](x) = \int_{S_1} k(x, y) f(y) \dd\mu_1(y) \qquad \mbox{ for } \mu_2\mbox{-almost every } x \in S_2.
    \]
    Buhvalov \cite{buhvalov} has characterized integral operators by the property that they map dominated, norm-convergent
    sequences to almost everywhere convergent sequences. 

    Interesting additional mapping properties can be characterized through integrability assumptions on the kernel. For example, a kernel operator
    $K$ maps $L^2(S_1)$ to $L^\infty(S_2)$ 
    if and only if $k\in L^\infty(S_2; L^2(S_1))$ in the sense that $\sup_{x\in S_2}\|k(x, \cdot)\|_{L^2(S_1)}<\infty$, see \cite[Theorem 1.3]{ab94}.
    In the case where $(S_1, \Sigma_1, \mu_1) = (S_2, \Sigma_2, \mu_2)$, it is well known that $K$ is a Hilbert--Schmidt operator 
    if and only if $k\in L^2(S_1\times S_1)$, see \cite[Theorem VI.6]{rs1}.
\end{example}

We now define the form $\fra\colon D(\fra)\times D(\fra) \to \C$ by setting 
\begin{equation}
    \label{eq.doma}
    D(\fra) = \{ \calu \in \calH \suchthat u_1 \in H^1(\Omega), u_2 = \trace u_1\}
\end{equation}
and then
\begin{align}
    \fra[\calu, \calv] &\coloneqq \frq[u_1, v_1] - \langle \calB\calu, \calv\rangle_\calH\\
    &= \frq[u_1, v_1] - \int_{\Omega} \big[B_{11}u_1 + B_{12}u_2\big]\overline{v_1}\dd\lambda - 
    \int_{\Gamma} \big[B_{21}u_1 + B_{22}u_2\big]\overline{v_2}\dd\sigma \notag
\end{align}
for $\calu,\calv \in D(\fra)$.

\begin{thm}
    \label{t.generate}
    Under Hypothesis~\ref{h.main}, the form $\fra$ is densely defined, closed and sectorial. The associated operator $\calA$ has compact resolvent, and $-\calA$
    generates an analytic, strongly continuous semigroup $(\calT(t))_{t\geq 0}$ on $\calH$. This semigroup is real in the sense that $\calT(t)\calH_\R \subset \calH_\R$ for all $t\geq 0$. 
\end{thm}

\begin{proof}
    We start by proving that $D(\fra)$ is dense in $\calH$. First note that $(\varphi, 0)\in D(\fra)$ whenever 
    $\varphi \in C_c^\infty(\Omega)$. It follows that $L^2(\Omega)\times\{0\} \subset \overline{D(\fra)}$. 
    Recall that the trace operator defines a bounded map from $H^1(\Omega)$ to $H^{1/2}(\Gamma)$ and the latter
    is dense in $L^2(\Gamma)$, see \cite[Theorem 3.38]{McLean00}. Thus, given $f_2\in L^2(\Gamma)$, we 
    find $u_1\in H^1(\Omega)$ with $\|f_2-\trace u_1\|^2 \leq \eps$.
    Now pick
    $\varphi\in C_c^\infty(\Omega)$ with $\|u_1-\varphi\|^2_\Omega\leq \eps$ and define $\calv=(v_1, v_2)$ by setting $v_1 = u_1 -\varphi$
    and $v_2 = \trace u_1 = \trace v_1$. Then $\calv \in D(\fra)$ and 
    \[
        \|(0,f_2) - \calv\|^2_\calH = \|v_1\|_\Omega^2 + \|\trace u_1 - f_2\|_\Gamma^2 \leq 2\eps.
    \]
    As $\eps>0$ was arbitrary, the claim follows. As $\overline{D(\fra)}$ is a vector space, it follows that
    $\overline{D(\fra)}=\calH$.\medskip

    To prove closedness, first observe that there exists $\tilde{\omega} \geq 0$, such that
    \[
    \Re \frq[u_1] + \tilde{\omega} \|u_1\|^2_\Omega \geq \eta\|u_1\|_{H^1}^2,
    \]
    see \cite[Equation (4.3)]{Ouhabaz}. In view of the boundedness of $\calB$, it follows that
    \[
    \Re \fra[\calu] + \omega\|\calu\|_\calH^2 \geq \eta \|u_1\|_{H^1}^2
    \]
    with $\omega = \tilde{\omega} + \|\calB\|$.
    On the other hand, we clearly have 
    \[
    |\fra [\calu,\calv]| \leq C \|u_1\|_{H^1}\|v_1\|_{H^1}
    \]
    for a constant $C$ that depends upon the $L^\infty$-bounds of the coefficients $a_{ij}, b_j,$ and $c_j$ for $i,j=1, \ldots, d$, and the operator
    norms $\|B_{kl}\|$ for $k,l=1,2$. This yields that the form $\fra$ is sectorial and that the associated norm
    \[
    \|\calu\|_\fra^2 := \Re\fra[\calu] + (\omega+ 1)\|\calu\|_{\calH}^2
    \]
    is equivalent to $\|u_1\|_{H^1}^2 + \|u_2\|_\Gamma^2$. Now let $\calu_n = (u_{1,n}, u_{2,n})$
    be a Cauchy sequence in $(D(\fra), \|\cdot\|_\fra)$. By what was done so far, $(u_{1,n})$ is a Cauchy sequence in $H^1(\Omega)$, hence convergent to some
    $u_1 \in H^1(\Omega)$, and $(u_{2,n})$ is a Cauchy sequence in $L^2(\Gamma)$ and thus has a limit, 
    say $u_2\in L^2(\Gamma)$. As the trace operator is continuous from $H^1(\Omega)$ to $L^2(\Gamma)$, it follows that $u_2=\trace u_1$,
    whence $\calu=(u_1, u_2) \in D(\fra)$. Thus, $\fra$
    is closed.

    Since $\Omega$ is a bounded Lipschitz domain, the embedding of $H^1(\Omega)$ into $L^2(\Omega)$ is compact. As $\trace : H^1(\Omega)\to L^2(\Gamma)$ is a continuous map, it follows that $\trace(H^1(\Omega))$ is a compact subset of $L^2(\Gamma)$. These observations imply that the embedding of $D(\fra)$ into $\calH$ is compact, and thus $\calA$ has compact resolvent.

    That the operator $-\calA$ generates an analytic, strongly continuous semigroup follows from general results concerning 
    densely defined, closed, sectorial forms, see Section 1.4 of 
    \cite{Ouhabaz}. To prove that the semigroup is real, we use a
    Beurling--Deny type criterion, see \cite[Proposition 2.5]{Ouhabaz}. We thus have to prove that for $\calu\in D(\fra)$,
    it holds that $\Re\calu \in D(\fra)$ and $\fra [\Re\calu, \Im\calu] \in \R$. 
    If $\calu = (u_1, u_2)\in D(\fra)$, then $\Re\calu = (\Re u_1, \Re u_2)$. 
    Since the trace operator is real, it follows that $\Re u_2 = \Re \trace u_1 = \trace \Re u_1$,
    proving $\Re\calu\in D(\fra)$. As all coefficients $a_{ij}, b_j, c_j$ are real-valued and the operators $B_{kl}$ are real,
    it easily follows that $\fra[\Re\calu, \Im\calu] \in \R$.
\end{proof}

We next identify the operator $\calA$ associated to the form $\fra$.

\begin{prop}
\label{p.domain}
    It holds that
    \[D(\calA)=\{\calu = (u_1, u_2) \suchthat u_1\in \Dm(L), \conormal u_1\in L^2(\Gamma), u_2 = \trace u_1\}\]
    and 
    \[\calA \calu= \begin{pmatrix}
        Lu_1 - B_{11}u_1 - B_{12}u_2\\ \conormal u_1 - B_{21}u_1 - B_{22}u_2
    \end{pmatrix}
    = \begin{pmatrix}
        Lu_1\\
        \conormal u_1
    \end{pmatrix} - \calB\calu.
    \]
\end{prop}

\begin{proof}
For the time being, let $\calA$ be the operator from the statement of the Proposition and $\calC$ be the operator associated to
the form $\fra$, i.e.\ $\calu \in D(\calC)$ with $\calC \calu = \calf$ if and only if $\calu\in D(\fra)$ and $\langle \calf, \calv\rangle_{\calH}
= \fra[\calu, \calv]$ for all $\calv\in D(\fra)$. We prove that $\calA = \calC$.\smallskip

We start by proving $\calA\subset \calC$. It obviously holds that $D(\calA)\subset D(\fra)$. Moreover, if $\calu \in D(\calA)$, 
then 
\begin{align*}
    \langle \calA\calu, \calv\rangle_\calH & = \langle Lu_1, v_1\rangle_\Omega + \langle \conormal u_1, v_2 \rangle_\Gamma
    -\langle \calB \calu, \calv\rangle_\calH\\
    & = \langle Lu_1, v_1\rangle_\Omega + \langle \conormal u_1, \trace v_1\rangle_\Gamma
    -\langle \calB \calu, \calv\rangle_\calH\\
    & = \frq[u_1, v_1] - \langle\calB\calu, \calv\rangle_\calH = \fra[\calu, \calv]
\end{align*}
for all $\calv\in D(\fra)$. This proves that $\calu \in D(\calC)$ and $\calC\calu =\calA\calu$.\smallskip

Conversely, assume that $\calu \in D(\calC)$ with $\calC\calu = \calf$. In this case,
\begin{align}
\frq[u_1, v_1] - \langle \calB\calu, \calv\rangle_\calH & =
\langle f_1, v_1\rangle_\Omega + \langle f_2, v_2\rangle_\Gamma\notag\\
& = \langle f_1, v_1\rangle_\Omega +\langle f_2, \trace v_1\rangle_\Gamma\label{eq.associated}
\end{align}
for all $\calv\in D(\fra)$. It follows from Lemma \ref{l.weaksol} that $u_1\in \Dm(L)$ and $Lu_1 = f_1 + B_{11}u_1 + B_{12}u_2$
and  $\conormal u_1 = f_2+B_{21}u_1+ B_{22}u_2\in L^2(\Gamma)$. It follows
that $f_2= \conormal u_1 - B_{21}u_1-B_{22}u_2$. Altogether, we have proved that
$\calu\in D(\calA)$ and $\calC\calu = \calA\calu$. This finishes the proof.
\end{proof}

\begin{rem}
Note that every $\calu \in D(\calA^2)$ satisfies the generalized Wentzell boundary condition 
\[ 
\trace (Lu_1-B_{11}u_1-B_{12}u_2)= \partial_\nu^Lu_1-B_{21}u_1-B_{22}u_2.
\]
Proposition \ref{p.domain} shows that $\trace (Lu_1-B_{11}u_1-B_{12}u_2)$ exists in $L^2(\Gamma)$. The individual traces of the summands
need not exist in general. We also point out that since $\calT$ is analytic, we have $\calT(t)\calu_0 \in D(\calA^k)$ for 
all $t>0$, $k \in \N$, and $\calu_0 \in \calH$.
\end{rem}

Following the ideas of \cite[Sections 3 and 4]{Nittka} we also obtain H\"older continuity of elements of $D(\calA^k)$ for
large enough $k$.

\begin{lem}
    \label{l.prep}
    Given $d\geq 3$, $\eps>0$, assume that $\calB$ is bounded on $L^{d-1+\eps}(\Omega)\times L^{d-1+\eps}(\Gamma)$.
    Moreover, let $2\leq p < d-1+\eps$ and let the function $\psi$ be as in Corollary \ref{c.imorove}.
    Then, if $\calu \in D(\calA)\cap (L^p(\Omega)\times L^p(\Gamma))$ with $\calA\calu \in L^p(\Omega)\times L^p(\Gamma)$, it follows that
    $\calu \in L^{\psi(p)}(\Omega)\times L^{\psi(p)}(\Gamma)$.
\end{lem}

\begin{proof}
    It follows by interpolation that $\calB$ maps $L^p(\Omega)\times L^p(\Gamma)$ to itself for every $p\in [2, d-1+\eps)$. 
    We note that for $\calu\in D(\calA)$, it holds that $u_1\in \Dm(L) \subset H^1(\Omega)$ and 
    $Lu_1 = (\calA\calu)_1 - B_{11}u_1 - B_{12}u_2$ and $\conormal u_1 = (\calA\calu)_2 -B_{21}u_1 - B_{22}u_2$. 
    Thus, if $\calu\in L^p(\Omega)\times L^p(\Gamma)$ and $\calA\calu\in L^p(\Omega)\times L^p(\Gamma)$, 
    it follows that $Lu_1\in L^p(\Omega)$ and $\conormal u_1 \in L^p(\Gamma)$. The claim now follows from Corollary \ref{c.imorove}.
\end{proof}

\begin{thm}
\label{t.hoelder}
Assume Hypothesis \ref{h.main} and let $\eps>0$. If $d \geq 3$, assume that $\calB$ is bounded on $L^{d-1+\eps}(\Omega) \times L^{d-1+\eps}(\Gamma)$. Then there are $k=k(d)\in \N$ and $\alpha \in (0,1)$ such that $\calu\in D(\calA^k)$ implies $u_1\in C^{\alpha}(\Omega)$, and hence $\mathcal{u} \in L^\infty(\Omega) \times L^\infty(\Gamma)$.
\end{thm}

\begin{proof}
If $\calu\in D(\calA)$, then $u_1\in \Dm(L) \subset H^1(\Omega)$ and \[Lu_1 = (\calA\calu)_1 - B_{11}u_1 - B_{12}u_2 \text{ and }
\conormal u_1 = (\calA\calu)_2 - B_{21}u_1 - B_{22}u_2.\] 

If $d=1$, then $H^1(\Omega)$ is continuously embedded into $C^{\alpha}(\Omega)$ and the result follows. In the case $d=2$, Lemma \ref{l.nittkahoelder}(ii), applied with $\eps=1$, implies that $u\in C^{\alpha}(\Omega)$ without additional regularity assumptions on $\calB$ as $d-1+\eps=2$.  

In the remaining case $d\geq 3$, it follows from Sobolev embedding that $u_1\in L^{\frac{2d}{d-2}}(\Omega)$ and $\trace u_1 \in L^{\frac{2d}{d-1}}(\Gamma)$.
It follows that $\calu \in L^{p_1}(\Omega)\times L^{p_1}(\Gamma)$ for $p_1= \frac{2d}{d-1}>2$. 
At this point, Lemma \ref{l.prep} and induction yield that $\calu\in D(\calA^k)$ implies $\calu \in L^{p_k}(\Omega)\times L^{p_k}(\Gamma)$, 
where $p_k \coloneqq \psi(p_{k-1})$, i.e. $p_k =\psi^{k-1}(p_1)$. It is clear from the structure of the function $\psi$ that $\psi^{k}(p) \to \infty$ as $k\to \infty$ 
for every $p\in (2, \infty)$. We thus find a smallest index $k^*$ such that $p_{k^*}\geq d-1+\eps$. Then for $u \in D(\calA^{k^*+1})$, we have $u, \calA u \in L^{p_{k^*}}(\Omega)\times L^{p_{k^*}}(\Gamma) \subset L^{d-1+\eps}(\Omega)\times L^{d-1+\eps}(\Gamma)$. 
It follows that also $\calB u \in L^{d-1+\eps}(\Omega) \times L^{d-1+\eps}(\Gamma)$. At this point Lemma \ref{l.nittkahoelder}(ii) yields $u_1\in C^{\alpha}(\Omega)$ for $\calu\in D(\calA^{k^*+1})$. 
\end{proof}

\section{Positivity and Markov properties}
\label{sec:positivity-markov}

Positivity is another key feature in this article, and therefore we recall some concepts regarding the order structure on $\calH$. The \emph{positive cone} in $\calH$ is
\begin{equation*}
    \calH_+ \coloneqq \{ \calu\in\calH_\R \suchthat u_1\ge 0, u_2\ge 0 \}.
\end{equation*}
Here and in what follows, $u_1\geq 0$ means $u_1(x) \geq 0$ for $\lambda$-almost every $x\in \Omega$ and $u_2\geq 0$ 
means $u_2(x)\geq 0$ for $\sigma$-almost every $x\in \Gamma$. 
We write $\calu\geq 0$ if $\calu\in \calH_+$. The notation $\calu>0$ indicates $\calu \geq 0$ and $\calu\neq 0$.
We say that $\calu$ is \emph{strictly positive}, and write $\calu\gg 0$, if $u_1(x)>0$ for $\lambda$-almost every
$x\in \Omega$ and $u_2(x)>0$ for $\sigma$-almost every $x\in \Gamma$.

The lattice operations of supremum and infimum in $\calH$ are denoted respectively by
\begin{equation*}
    \calu\vee \calv \coloneqq \sup(\calu,\calv), \qquad \calu\wedge \calv \coloneqq \inf(\calu, \calv),
\end{equation*}
for all $\calu,\calv\in\calH$, and should be interpreted component-wise; for instance,
\begin{equation*}
    \calu \vee \calv = (u_1\vee v_1, u_2\vee v_2).
\end{equation*}
Moreover, we define the \emph{positive} and \emph{negative parts} of an element $\calu\in\calH$ by
\begin{equation*}
    \calu^+ \coloneqq \calu\vee 0, \qquad \calu^- \coloneqq (-\calu)\vee 0,
\end{equation*}
and the \emph{modulus} is given by
\begin{equation*}
    |\calu| \coloneqq \calu \vee (-\calu) = (|u_1|, |u_2|).
\end{equation*}

An operator $\calS$ on $\calH$ is called \emph{positive} if $\calS(\calH_+)\subset \calH_+$. Observe that a positive operator is automatically real. We denote the constant function on $\calH$ with value 1 by $\one = (\one_\Omega, \one_\Gamma)$. A positive operator
$\calS$ on $\calH$ is called Markovian, if $\calS\one = \one$; it is called (\emph{sub-})\emph{Markovian} if $\calS\one\leq \one$.
We call a semigroup $\calS = (\calS(t))_{t\geq 0}$ positive ((sub-)Markovian) if every operator $\calS(t)$ is positive ((sub-)Markovian).\smallskip

Before characterizing positivity and (sub-)Markovianity of the semigroup $\calT$ associated to the form $\fra$, we recall the following notion.

\begin{defn}
    If $(M,\mu)$ is a measure space and $H= L^2(M, \mu)$, then 
    a bounded, real linear operator $S: H \to H$ is said to \emph{satisfy the positive minimum principle} if
    $\langle Sf, g\rangle \geq 0$ whenever $f, g \in H_+$ satisfy $\langle f,g\rangle = 0$.
\end{defn}

The importance of operators that satisfy the positive minimum principle stems from the following result, which is taken from \cite[Theorem C-II.1.11]{schreibmaschiene}, where it is stated in the general setting of Banach lattices.

\begin{lem}
    \label{l.pmp}
    Let $A$ be a real, bounded linear operator on $L^2(M, \mu)$. The following are equivalent:
    \begin{enumerate}[label=\upshape(\roman*)]
        \item $e^{tA} = \sum_{k=0}^\infty \frac{(tA)^k}{k!} \geq 0$ for all $t\geq 0$;
        \item $A$ satisfies the positive minimum principle;
        \item $A+\|A\|I \geq 0$.
    \end{enumerate}
\end{lem}

We point out that it is equivalent to ask in (iii) that there is some $\alpha \geq 0$ such that $A+\alpha I\geq 0$, as the semigroup generated
by $A+\alpha I$ is given by $(e^{\alpha t}e^{tA})_{t\geq 0}$ and this is positive if and only if the semigroup $(e^{tA})_{t\geq 0}$ is positive.
It follows that a bounded linear operator $A$ on an $L^2$-space satisfies the positive minimum principle if and only if it can be written as
$A=P-M$, where $P$ is a positive operator and $M$ is a multiplication operator, i.e.\ 
$[Mf](x) = m(x)f(x)$ for some function $m\in L^\infty(M)$. 
This representation is the appropriate generalization of a matrix with positive off-diagonal entries to the $L^2$-setting.
An important special case is when $P$ is a positive integral operator and $M$ is chosen in such a way that $A\one=0$.

\begin{example}
    Let $0\leq k \in L^\infty(\Omega; L^2(\Omega))$ and put $\mu(x) = \int_\Omega k(x, y)\dd y$. Then the operator $A\in \calL(L^2(\Omega))$,
    defined by
    \[
    [Af](x) = \int_\Omega k(x,y)[f(y) - f(x)]\dd y = \int_\Omega k(x,y)f(y)\dd y - \mu(x)f(x)
    \]
    satisfies the positive minimum principle.
\end{example}

We can now characterize positivity and (sub-)Markovianity of the semigroup $\calT$.
\begin{thm}\label{t.characterization}
Assume Hypothesis \ref{h.main}. For parts {\upshape(b)} and {\upshape(c)} additionally assume that $c \in W^{1,\infty}(\Omega;\R^d)$. We denote the unit outer normal of $\Omega$ by $\nu$.
\begin{enumerate}
[\upshape (a)]
    \item The semigroup $\calT$ is positive if and only if
    \begin{enumerate}[\upshape (i)]
        \item $B_{12}$ and $B_{21}$ are positive operators;
        \item $B_{11}$ and $B_{22}$ satisfy the positive minimum principle.
    \end{enumerate}
    \item The semigroup $\calT$ is sub-Markovian if and only if
    \begin{enumerate}
    [\upshape (i)]
        \item $\calT$ is positive, i.e.\ conditions {\upshape(i)} and
        {\upshape(ii)} of part {\upshape(a)} are satisfied,
        \item $\Div c + B_{11}\one_\Omega + B_{12}\one_\Gamma \leq 0$ and
        \item $B_{21}\one_\Omega + B_{22}\one_\Gamma \leq c\cdot\nu$.
    \end{enumerate}
    \item The semigroup $\calT$ is Markovian if and only if 
    conditions {\upshape(i) -- (iii)} from {\upshape(b)} are satisfied with equality in {\upshape(ii)} and {\upshape(iii)}.
\end{enumerate}
\end{thm}

\begin{proof}
    (a) As $\calT$ is real, it follows from \cite[Theorem 2.6]{Ouhabaz} that $\calT$ is positive if and only if for every real-valued
    $\calu \in D(\fra)$, it holds that $\calu^+\in D(\fra)$ and $\fra [\calu^+, \calu^-] \leq 0$. An easy calculation shows that 
    (i) and (ii) are sufficient for the positivity of $\calT$. 

    To prove that they are necessary, assume that $\calT$ is positive and let $\calu =(u_1, u_2) \in D(\fra)$ be real. 
    Note that $u_1^+ \in H^1(\Omega)$ whenever $u_1\in H^1(\Omega)$ is real. In this case, we also have $(\trace u_1)^+ = \trace (u_1^+)$, so that   $\calu^+ = (u_1^+,u_2^+)$ belongs to $D(\fra)$ whenever $\calu \in D(\fra)$.

    Next, we recall Stampacchia's lemma~\cite[Lemma 7.6]{GT}, which states that $D_j u_1^+ = \one_{\{u_1>0\}}D_ju_1$ and $D_ju_1^- = \one_{\{u_1<0\}} D_ju_1$. It follows that
    $\frq [u_1^+, u_1^-]=0$ for all $u\in H^1(\Omega)$. Thus, \cite[Theorem 2.6]{Ouhabaz} implies that
    \begin{equation}\label{eq.positivebd}
    0\leq \langle \calB\calu^+, \calu^-\rangle_\calH. 
    \end{equation}
    It follows from Theorem \ref{t.generate} that, given $f_1\in L^2(\Omega)_+$ and $f_2\in L^2(\Gamma)_+$, 
    we find a sequence $(\calu_n)\subset D(\fra)$ with $\calu_n \to (f_1, - f_2)$ in $\calH$.
    By continuity of the lattice operations, $(u_{1,n})^+ \to f_1$ and $(u_{1,n})^- \to 0$ in $L^2(\Omega)$ and $(u_{2,n})^+ \to 0$,
    $(u_{2,n})^- \to f_2$ in $L^2(\Gamma)$. Thus, using  \eqref{eq.positivebd} with $\calu=\calu_n$, it follows upon letting $n\to \infty$
    that 
    \[
    0\leq \int_\Gamma (B_{21}f_1) f_2\dd \sigma.
    \]
    As $f_1$ and $f_2$ are arbitrary, the positivity of $B_{21}$ follows. The positivity of $B_{12}$ is proved similarly, approximating
    $(-f_1, f_2)$ instead. This proves condition (i).

    Approximating $(f_1, 0)$ for arbitrary $f_1 \in L^2(\Omega)$, \eqref{eq.positivebd} yields
    $\langle B_{11} f_1^+, f_1^-\rangle_\Omega \geq 0$. Given $f,g\in L^2(\Omega)_+$ with $\langle f, g\rangle_\Omega = 0$, we may consider
    $f_1 = f-g$, so that $f_1^+=f$ and $f_1^-=g$, to infer $\langle B_{11}f,g\rangle_\Omega \geq 0$. This proves that $B_{11}$ satisfies
    the positive minimum principle. To establish the positive minimum principle for $B_{22}$, we approximate $(0, f_2)$ for 
    arbitrary $f_2\in L^2(\Gamma)$ and argue similarly.\medskip

    (c) Let us consider the case where $\calT$ is positive. In this case, $\calT$ is Markovian if and only if for
    $\one = (\one_\Omega, \one_\Gamma)$ we have $\calT(t)\one = \one$ for all $t\geq 0$ which, in turn, is equivalent to $\one\in \ker(-\calA)$.
    
We note that $\one_\Omega\in \Dm(L)$ with  $L\one_\Omega = -\Div c$ and
 $\conormal\one_\Omega = c\cdot\nu$. It thus follows from
    Proposition \ref{p.domain} that $\one\in \ker(-\calA)$ if and only if $\Div c + B_{11}\one_\Omega + B_{12}\one_\Gamma =0$
    and $B_{21}\one_\Omega + B_{22}\one_\Gamma = c\cdot\nu$.\medskip

    (b) To prove necessity of (i) -- (iii), assume that $\calT$ is sub-Markovian. Then $\calT$ is positive and (i) follows from part (a).  
    By the Beurling--Deny criterion (see \cite[Corollary 2.17]{Ouhabaz}) $\calT$ is sub-Markovian
    if and only if for every $\calu\in D(\fra)\cap\calH_+$ it holds that $\calu\wedge \one$ in $D(\fra)$ and
    $\fra[\calu\wedge\one, (u-\one)^+]\geq 0$.
    Noting that $D_j(u_1\wedge \one_\Omega) = \one_{\{u_1< 1\}}D_j u$ and $D_j(u_1-\one_\Omega)^+ = \one_{\{u_1>1\}}D_j u$, it follows
    that 
    \begin{align*}
        0 & \leq \fra[\calu\wedge \one, (\calu-\one)^+]\\
        & = \sum_{j=1}^d \int_\Omega c_j D_j(u_1-\one_\Omega)^+\dd \lambda - \int_\Omega \big[B_{11}(u_1\wedge \one_\Omega) + B_{12}(u_2\wedge \one_\Gamma)\big](u_1 - \one_\Omega)^+\dd \lambda\\
        & \qquad - \int_\Gamma \big[B_{21}(u_1\wedge \one_\Omega) + B_{22}(u_2\wedge\one_\Gamma)\big](u_2-\one_\Gamma)^+\dd \sigma.
    \end{align*}
    Integrating by parts in the first integral and inserting $\calu = \calv +\one$ for $0\leq \calv \in D(\fra)$, it follows that 
    \[
    \int_\Omega \big(B_{11}\one_{\Omega} + B_{12}\one_{\Gamma} + \Div c \big)v_1\dd\lambda +\int_\Gamma \big(
    B_{22}\one_\Gamma + B_{21}\one_{\Omega} - c\cdot\nu \big)v_2\dd\sigma \leq 0.
    \]
    By density, this inequality extends to arbitrary $\calv \in \calH_+$ and this proves the necessity of conditions (ii) and (iii).\smallskip    

    It remains to prove the sufficiency of conditions (i) -- (iii). We observe that, in view of part (a), (i) immediately implies that
    $\calT$ is positive. We now employ a technical construction. 
    Note that the orthogonal projection onto the linear span of $\one_\Gamma$ is given by $P_\Gamma u \coloneqq \sigma(\Gamma)^{-1} \langle u, \one_\Gamma\rangle_\Gamma \one_\Gamma$. We define a new operator $\tilde{B}_{12} \in \calL(L^2(\Gamma), L^2(\Omega))$ by
    \[
    \tilde{B}_{12}u \coloneqq B_{12}(I-P_\Gamma)u- \frac{1}{\sigma(\Gamma)}\langle u, \one_\Gamma\rangle_\Gamma \big( B_{11}\one_\Omega + \Div c \big).
    \]
    Setting $u = \one_\Gamma$ yields $\Div c + B_{11}\one_\Omega + \tilde{B}_{12}\one_\Gamma = 0$. 
    If $0\leq u \in L^2(\Gamma)$, then
    $\langle u, \one_\Gamma\rangle \geq 0$ and condition (ii) yields
    \[
    \tilde{B}_{12} u \geq B_{12}(I-P_\Gamma)u + \frac{1}{\sigma(\Gamma)} \langle u, \one_\Gamma\rangle_\Gamma B_{12}\one_\Gamma  = B_{12}u.
    \]
    This proves that $B_{12} \leq \tilde{B}_{12}$. In particular, as $B_{12}$ is positive, so is $\tilde{B}_{12}$.
    Similarly, for every $u\in L^2(\Omega)$, we consider the orthogonal projection $P_\Omega u:=\lambda(\Omega)^{-1}\langle u,\one_\Omega\rangle_\Omega \one_\Omega$ and define
    \begin{equation*}
        \tilde{B}_{21}u \coloneqq B_{21}(I-P_\Omega)u - \frac{1}{\lambda(\Omega)}\langle u,\one_\Omega\rangle \big(B_{22}\one_\Gamma - c\cdot\nu \big)
    \end{equation*}
    for all $u\in L^2(\Omega)$. One checks as above that $\tilde{B}_{21} \in \calL(L^2(\Omega), L^2(\Gamma))$ is a positive operator such that
    $\tilde{B}_{21}\one_\Omega + B_{22}\one_\Gamma = c\cdot\nu$ and $B_{21}\leq \tilde{B}_{21}$.   

    Now consider the operator
    \[
    \tilde{\calB} = \begin{pmatrix}
        B_{11} &  \tilde{B}_{12}\\
        \tilde{B}_{21} & B_{22}
    \end{pmatrix}
    \]
    and define $\tilde{\fra}[\calu,\calv] \coloneqq \frq[\calu, \calv] - \langle \tilde\calB\calu, \calv\rangle_\calH$
    for $\calu, \calv\in D(\tilde{\fra}) \coloneqq D(\fra)$. 
    It follows from part (c) that the semigroup $\tilde{\calT}$ associated with $\tilde{\fra}$ is Markovian. It is straightforward to check that
    $\tilde{\fra}[\calu, \calv] \leq \fra[\calu, \calv]$ for all $0\leq \calu, \calv \in D(\fra)$. Thus, by the Ouhabaz domination criterion
    for positive semigroups (see \cite[Theorem 2.2.4]{Ouhabaz}), it follows that
    \[
    0 \leq \calT(t)\calf \leq \tilde{\calT}(t)\calf \qquad \mbox{ for all } t\geq 0, \calf\in \calH_+.
    \]
    As $\tilde{\calT}$ is Markovian, this clearly implies that $\calT$ is sub-Markovian.
\end{proof}

\begin{rem}
    (i) The assumption that $c \in W^{1,\infty}(\Omega;\R^d)$ in Theorem \ref{t.characterization} is necessary for 
    parts (b) and (c). Indeed, \cite[Example 4.4]{Nittka} provides an example that without this assumption the semigroup $\calT$ (even after possible rescaling) is not contractive on $L^\infty(\Omega)\times L^\infty(\Gamma)$ and thus, in particular, not sub-Markovian. 
    (ii) The conditions of Theorem \ref{t.characterization}(a) are equivalent to the positivity of the semigroup $(e^{t\calB})_{t\geq 0}$ 
    on $\ltwoog$ and thus to the operator $\calB$ satisfying the positive minimum principle. 
\end{rem}

We end this section by discussing \emph{irreducibility} of the semigroup $\calT$. If $E$ is a Banach lattice, then a subspace $J$ of
$E$ is called an \emph{ideal} if
\begin{enumerate}[(i)]
    \item $u\in J$ implies $|u|\in J$; and
    \item if $0\leq v\leq u$ and $u\in J$, then also $v\in J$.
\end{enumerate}
A strongly continuous semigroup on $E$ is called \emph{irreducible} if the only closed ideals that are invariant under the semigroup are 
$\{0\}$ and $E$. Often, an irreducible semigroup is tacitly assumed to be positive. This is the case, for example, in~\cite[Section C-III.3]{schreibmaschiene}, where one can find a characterization of irreducibility for strongly continuous positive semigroups on Banach lattices. However, if the semigroup is positive and analytic, 
as is the case when the semigroup arises from a form, then irreducibility is equivalent to the formally stronger
notion of \emph{positivity improving} in the sense that $f>0$ implies $S(t)f \gg 0$ for all $t>0$; cf.\ \cite[Definition 2.8 \& Theorem 2.9]{Ouhabaz} or~\cite[C-III Theorem 3.2(b)]{schreibmaschiene}. This type of result has recently been shown to hold for \emph{eventually} positive semigroups in~\cite{AG24} (see Proposition 3.12 in particular), where the reader will find a more thorough investigation of irreducibility under eventual positivity assumptions.
We stress, however, that in our terminology, an irreducible semigroup is \emph{not} assumed to be positive (or even eventually positive), in general.

If $E=L^p(M)$ ($1\leq p <\infty)$, then $J\subset E$ is a closed ideal if and only if there is a measurable subset $S\subset M$ such that
\[
J = \{f\in L^p(M) \suchthat f|_{S} = 0\; \text{a.e.}\};
\]
see for instance~\cite[Proposition 10.15]{BFR17}. In order to make use of this characterization in our setting, we will identify
$L^2(\Omega)\times L^2(\Gamma)$ with $L^2(\oug, \lambda\otimes\sigma)$ as before.

\begin{prop}
    \label{p.irreducible}
    If $\Omega$ is connected then the semigroup $\calT$ is irreducible.
\end{prop}

\begin{proof}
    Given a measurable subset $S$ of $\oug$, we identify $L^2(S)$ with the closed subspace 
    \[
    \{ f\in L^2(\oug) \suchthat f|_{S^c}=0\quad \text{a.e.}\}.
    \]
     To establish irreducibility, we have to prove that $\calT(t)L^2(S)\subset L^2(S)$ for all $t>0$ implies $(\lambda\otimes\sigma)(S)=0$
    or $(\lambda\otimes \sigma)((\oug)\setminus S)=0$. To that end, we can use a Beurling--Deny type criterion, see \cite[Theorem 2.10 and Corollary 2.11]{Ouhabaz}.
    We point out that the assumption of accretivity in that Theorem is not needed, as we may rescale the semigroup appropriately.
    It thus suffices to prove that if $S \subset \oug$ satisfies $\one_S\calu\in D(\fra)$ for all $\calu\in D(\fra)$, then
    $(\lambda\otimes\sigma)(S) =0$ or $(\lambda\otimes\sigma)((\oug)\setminus S) = 0$. Assume that $(\lambda\otimes\sigma)(S)>0$.
    Note that if $\calu \in D(\fra)$ satisfies $u_1=0$ almost everywhere then $\calu =0$. This implies that $S_1 \coloneqq S\cap \Omega$ has positive Lebesgue measure. It follows that for every $u_1\in C_c^\infty(\Omega)$ we have $\one_{S_1}u_1\in H^1(\Omega)$.
    Arguing as in the proof of \cite[Theorem 4.5]{Ouhabaz}, we see that this is only possible if $\lambda (\Omega\setminus S_1) = 0$.
    It now follows that also $(\lambda\otimes\sigma)((\oug)\setminus S) = 0$.
\end{proof}

\section{The semigroup on the space of continuous functions}

Under the assumptions of Theorem \ref{t.hoelder} the semigroup $\calT$ maps $\ltwoog$ into $\linftyog$. In particular, 
we may consider the restriction $\calT_\infty$ of $\calT$ to $\linftyog$. We prove that, under appropriate assumptions,
this restriction is a strong Feller semigroup in the sense of Definition \ref{def.sf}. Our setting is as follows.

As a compact space, we choose $M = \overline{\Omega}\times\Gamma$, endowed with the product topology and let $\mu = \lambda\otimes\sigma$.
Then $\mu(B(x,\eps))>0$ for all $x\in M$ and $\eps>0$. Moreover, we may then identify $\ltwoog$ with $L^2(M, \mu)$ and $\linftyog$ with $L^\infty(M, \mu)$. Likewise, the space $C(M)$ can be identified with $C(\overline{\Omega})\times C(\Gamma)$. We will consider the space
\[
\calC \coloneqq \{ \calu \suchthat u_1\in C(\overline{\Omega}) \mbox{ with } u_2 = \trace u_1\} \subset C(\overline{\Omega})\times C(\Gamma),
\]
which is obviously closed in $C(\overline{\Omega})\times C(\Gamma)$.
We start with addressing the situation where $\calB=0$. We denote the form $\fra$ with $\calB=0$ by $\frh$ and the associated operator by $\calL$. The semigroup generated by $-\calL$ is denoted by $\calS$.

\begin{prop} 
    \label{p.Lstrongfeller}
    Assume Hypothesis \ref{h.prelim} and additionally $c\in W^{1,\infty}(\Omega; \R^d)$. Then, the semigroup $\calS$ restricts to a weak$^*$-semigroup $\calS_\infty$
    in the sense of Definition \ref{def.ws} on the space $\linftyog$. The semigroup $\calS_\infty$ is a strong Feller semigroup with
    respect to $\calC$. In particular, it restricts to a strongly continuous semigroup $\calS_\calC$ on $\calC$.
\end{prop}

\begin{proof}
    We write $\calS_\infty(t)\coloneqq \calS(t)|_{\linftyog}$, which is well-defined by Theorem~\ref{t.hoelder}. 
    To prove that $\calS_\infty(t)$ is an adjoint operator, we use
    Lemma \ref{l.adjoint}. If $(\calf_n)_{n\in \N}$ is a bounded sequence in $\linftyog$ that converges pointwise to $\calf$, then
    $\calf_n \to \calf$ in $\ltwoog$ by dominated convergence. It follows that $\calS(t)\calf_n \to \calS(t)f$ in $\ltwoog$. 
    Passing to a subsequence we may (and shall) assume that $\calS_\infty(t)\calf_n\to \calS_\infty(t)\calf$ pointwise almost everywhere.  
    As $\calS_\infty(t)$ is a bounded operator on $\linftyog$, the sequence $\calS_\infty (t)\calf_n$ is uniformly bounded.
    Thus, if $\calg\in L^1(\Omega)\times L^1(\Gamma)$, it follows by dominated convergence that $\langle \calg, \calS_\infty(t)\calf_n\rangle \to 
    \langle \calg, \calS_\infty(t)\calf\rangle$. As this is true for every subsequence, it follows that $\calS_\infty(t)\calf_n\weakstar
    \calS_\infty(t)\calf$ as $n\to\infty$. By Lemma \ref{l.adjoint}, $\calS_\infty(t)$ is an adjoint operator.

    In order to prove that $\calS_\infty$ is an adjoint semigroup, it remains to show that $\calS_\infty(t)\calf \weakstar \calf$ for every $\calf\in \linftyog$ as $t \searrow 0$. To that end, we first note that since $c\in W^{1,\infty}(\Omega; \R^d)$, it follows from \cite[Proposition 4.5]{Nittka}, 
    that there exists $\alpha \geq 0$ such that $\|\calS_\infty(t)\| \leq e^{\alpha t}$ for all $t\geq 0$.
    By strong continuity on $\ltwoog$, for every $\calf\in \linftyog$ we have $\calS_\infty (t)\calf \to \calf$
    in $\ltwoog$. 
    Thus, given a sequence  $t_n\searrow 0$, we may assume, passing to a subsequence, that $\calS_\infty(t_n)\calf \to \calf$
    pointwise almost everywhere. Using dominated convergence, it follows that
    \[
    \langle \calg, \calS_\infty(t_n)\calf\rangle \to \langle \calg, \calf\rangle
    \]
    for every $\calg\in L^1(\Omega)\times L^1(\Gamma)$. As this is true for every subsequence, it follows that $\calS_\infty(t)\calf \weakstar \calf$ as $t\searrow 0$. This proves that $\calS_\infty$ is a weak$^*$-semigroup. Clearly the generator of $\calS_\infty$ is
    $-\calL_\infty$, the part of $-\calL$ in $\linftyog$.\medskip

    It remains to prove that $\calS_\infty$ is a strong Feller semigroup with respect to $\calC$. To that end, we first note that Theorem
    \ref{t.hoelder} and the analyticity of $\calS$ imply that $\calS_\infty$ maps $\linftyog$ to $\calC$. It follows from \cite[Lemma 4.6]{Nittka} that
    the domain of $-\calL_\infty$ is dense in $\calC$, whence \cite[Corollary 3.3.11]{abhn} implies that $\calS_\infty$ restricts 
    to a strongly continuous semigroup $\calS_\calC$ on $\calC$. This finishes the proof.
\end{proof}

We now turn to the semigroup $\calT$. In order to establish that also $\calT$ restricts to a strong Feller semigroup with respect
to $\calC$, we employ Theorem \ref{t.perturb} and make an additional assumption on $\calB$. 

\begin{hyp}
    \label{h.additional}
    We assume $c\in W^{1,\infty}(\Omega; \R^d)$ and that the operator $\calB$ maps $L^\infty(\Omega) \times L^\infty(\Gamma)$ to itself.
\end{hyp}

\begin{thm}
    \label{t.sgonc}
     Assume in addition to Hypothesis \ref{h.main} also Hypothesis \ref{h.additional}.
     Then $\calT_\infty\coloneqq \calT|_{\linftyog}$ is a strong Feller semigroup with respect to $\calC$. In particular,
     it restricts to a strongly continuous semigroup $\calT_\calC$ on $\calC$. The generator of $\calT_\calC$
     is $-\calA_\calC$, the part of $-\calA$ in $\calC$.
\end{thm}

\begin{proof}
    Once again, let $\calL$ denote the operator associated to the form $\frh$ and let $\calS$ be the semigroup generated by $-\calL$. By Proposition \ref{p.Lstrongfeller}, $\calS_\infty\coloneqq \calS|_{\linftyog}$
    is a strong Feller semigroup with respect to $\calC$. We denote its weak$^*$-generator by $-\calL_\infty$ as above.

    Similar arguments as in the proof of Proposition \ref{p.Lstrongfeller} show that $\calB|_{\linftyog}$ is an 
    adjoint operator. Thus, Theorem \ref{t.perturb} yields that $-\calL_{\infty}+\calB|_{\linftyog}$
    is the weak$^*$-generator of a strong Feller semigroup with respect to $\calC$. 
    Noting that $-\calL_{\infty}+\calB|_{\linftyog}$ is merely the part of $-\calA$ in $\linftyog$, 
    it follows from the uniqueness theorem for Laplace transforms, that
    the semigroup generated by $-\calL_{\infty}+\calB|_{\linftyog}$ must be the restriction of $\calT$ to $\linftyog$.
\end{proof}

\begin{rem}
    \label{r.similar}
    Define $V : C(\overline{\Omega}) \to \calC$ by $Vu= (u, \trace u)$. 
    Then $V$ is bijective with inverse $V^{-1}: (u, \trace u) \mapsto u$. Instead of the 
    semigroup $\calT_\calC$, it is often preferable to consider the \emph{similar} semigroup $T_C\coloneqq V^{-1}\calT_\calC V$ on $C(\overline{\Omega})$. It is again a 
    strongly continuous semigroup and its generator is $-A_C$ where 
    \begin{align*}
     D(A_C) & = \big\{ u \in \Dm (L)\cap C(\overline{\Omega}) \suchthat \conormal u \in L^2(\Gamma), 
     Lu - B_{11}u - B_{12} \trace u \in C(\overline{\Omega}),\\
     & \mbox{ and } \trace \big(Lu - B_{11}u - B_{12}\trace u\big) = \conormal u - B_{21}u - B_{22}\trace u\big\}
     \end{align*}
     and $A_Cu= Lu - B_{11}u - B_{12}\trace u$. 
     Thus, elements of $D(A_C)$ satisfy the Wentzell--Robin boundary condition. We note that if $\calB$ maps $C(\overline{\Omega})\times C(\Gamma)$
     to itself, then it follows that for $u\in D(A_C)$ we have $Lu\in C(\overline{\Omega})$ and $\conormal u \in C(\Gamma)$, 
     cf.\ \cite[Theorem 3.3]{ampr03}, and the boundary condition also holds in a pointwise sense. 
\end{rem}

\section{Spectral theory and asymptotic behavior}

In this section, we study the spectrum of $-\calA$, the generator
of the semigroup $\calT$ on $\calH = L^2(\Omega)\times L^2(\Gamma)$.
If Hypothesis \ref{h.additional} is satisfied, we may also
consider the semigroup $\calT_\calC$ (with generator $-\calA_\calC$) on the space $\calC$, 
or the semigroup $T_C$ (with generator $-A_C$) on the space $C(\overline{\Omega})$. We note
that $-A_C$ and $-\calA_\calC$ are similar, see Remark \ref{r.similar}, so that $\sigma(-\calA_\calC) = \sigma(-A_C)$.

We start with a general result. Recall that the \emph{spectral bound}
$\spb(A)$ is defined by
\[
\spb(A) = \sup\{\Re\lambda \suchthat \lambda \in \sigma(A)\}.
\]
Given a semigroup $T=(T(t))_{t\geq 0}$ the \emph{growth bound}
$\omega_0(T)$ is defined by
\[
\omega_0(T) = \inf\{ \omega\in \R \suchthat \exists M\geq 0 \mbox{ with }
\|T(t)\|\leq M e^{\omega t} \mbox{ for all } t\geq 0\}.
\]

\begin{prop}
\label{p.spectrum}
Assume Hypothesis \ref{h.main}.
\begin{enumerate}[\upshape (a)]
    \item $\sigma(-\calA)$ consists of only isolated eigenvalues which are poles of the resolvent and whose eigenspaces are finite dimensional. 
    \item  Assume additionally Hypothesis \ref{h.additional}. Then $\sigma(-\calA) = \sigma(-\calA_C)=\sigma(-A_C)$
    and all spectra only consist of isolated eigenvalues with finite-dimensional eigenspaces.
    \item It holds that $\spb(-\calA) = \omega_0(\calT)$. Furthermore, assuming additionally Hypothesis \ref{h.additional},
    we also have $\spb(-\calA_\calC)= \omega_0(\calT_\calC)$ and $\spb(-A_C)=\omega_0(T_C)$.
\end{enumerate}
\end{prop}

\begin{proof}
    As $-\calA$ has compact resolvent by Theorem \ref{t.generate}, part (a) follows immediately from \cite[Theorem III.6.29]{Kato76}.
    Now additionally assume Hypothesis \ref{h.additional} is satisfied.
    We note that the semigroup $\calT_\calC$ is also compact, as it maps (by analyticity of $\calT$ and Theorem \ref{t.hoelder})
    into the space $\calC^\alpha \coloneqq \{ (u, \trace u) \suchthat u\in C^{\alpha}(\overline{\Omega})\}$, which is easily
    seen to be a compact subset of $\calC$. Compactness of $T_C$ now follows from similarity. It also follows that
    $\sigma (-A_C)$ and $\sigma (-\calA_\calC)$ consist of isolated eigenvalues only with finite-dimensional eigenspaces.

    As the semigroups $\calT$ and $\calT_\calC$ are consistent and compact, \cite[Proposition 2.6]{arendt94} yields
    $\sigma(-\calA) = \sigma(-\calA_\calC)$. That $\sigma(-\calA_\calC) = \sigma(-A_C)$ follows by similarity.  At this point, (b) is proved.

    As for part (c), we note that since all the semigroups are compact, they are immediately norm continuous and hence the equality of growth and spectral bounds 
    follows from \cite[Corollary IV.3.11]{en00}.
\end{proof}

In the rest of this section, we take a closer look at the spectral bound $\spb(-\calA)$. We are particularly interested in 
the question whether $\spb(-\calA)\in \sigma(-\calA)$ and, if this is the case, in additional information about this spectral value. 
We briefly recall the relevant terminology. Given a closed operator $A$, we say that $\spb(A)$ is a \emph{dominant eigenvalue}, if
$\spb(A)$ is an eigenvalue of $A$ (thus, in particular, $\spb(A) \in \sigma(A)$) and $\Re\lambda < \spb(A)$ for all $\lambda \in \sigma(A)\setminus \{\lambda\}$. Note that if $\spb(A)$ is an eigenvalue of $A$, then $\spb(A)$ is dominant if and only if $\sigma(A) \cap (\spb(A) + i\R) = \{\spb(A)\}\subset \R$.

We call an eigenvalue $\lambda_0$ of an operator $A$ \emph{algebraically simple} if it is an isolated point of the spectrum 
and the associated spectral projection $P$, defined by
\begin{equation*}
    \label{eq.spectralprojection}
    P \coloneqq \frac{1}{2\pi i}\int_{|\lambda-\lambda_0| = \eps} R(\lambda, A)\, d\lambda,
\end{equation*}
where $\eps>0$ is chosen small enough so that $\bar{B}(\lambda_0, \eps)\cap\sigma(A) = \{\lambda_0\}$, has rank 1. We note that
if $\lambda_0$ is algebraically simple, then $\lambda_0$ is a first order pole of the resolvent and the eigenspace $\ker (A-\lambda_0)$ is 
one-dimensional. Moreover, the \emph{generalized eigenspace} $\bigcup_{n\in \N} \ker (A-\lambda_0)^n$ is also one-dimensional. It is well-known that if $\lambda_0$ is a pole of the resolvent (of any order), then $\lambda_0$ is algebraically simple if and only if the generalized eigenspace is one-dimensional. Moreover, if $\lambda_0$ is an algebraically simple eigenvalue, it is a \emph{geometrically simple} eigenvalue,
i.e.\ $\ker (A-\lambda_0)$ is one-dimensional. Conversely, if $\lambda_0$ is geometrically simple, then $\lambda_0$ is algebraically simple 
if and only if $\lambda_0$ is a pole of first order of the resolvent, which in turn is equivalent to the property $\ker(A-\lambda_0) = \ker((A-\lambda_0)^2)$.
For more information, we refer to \cite[Section III.6]{Kato76}, or~\cite{CampbellDaners}.

\subsection{The case of positive semigroups}
\label{subsec.positive-semigroup}

Throughout this section, we assume that Hypothesis \ref{h.additional} is satisfied and, moreover, that the semigroup
$\calT$ (and thus, by consistency, also $\calT_{\calC}$ and $T_C$) are positive. The latter is characterized by 
Theorem \ref{t.characterization}(a). Our primary goal is to describe the asymptotic behavior of these semigroups. 

Before we state and prove the main theorem  of this section, we recall a result about the strict monotonicity of the spectral 
bound, that we will use in the proof.

\begin{thm}
\label{t.monotone}
    Assume that $\calS_1, \calS_2$ are strongly continuous semigroups on a Banach lattice $E$ with generators $-\calA_1$ and $-\calA_2$,
    respectively. Assume that
    \begin{enumerate}[\upshape (i)]
        \item $0\leq \calS_1(t)\leq \calS_2(t)$ for all $t\geq 0$,
        \item $\calA_2$ has compact resolvent,
        \item $\calS_2$ is irreducible.
    \end{enumerate}
    Then, if $\calA_1\neq \calA_2$, we have $\spb(-\calA_1)< \spb(-\calA_2)$.
\end{thm}

\begin{proof}
    This is a version of \cite[Theorem 1.3]{ab92}.
\end{proof}

We can now characterize the asymptotic behavior of the semigroup $T_C$ on $C(\overline{\Omega})$. Using the results of 
Proposition \ref{p.spectrum}, it is straightforward to see that similar results also apply to the semigroups $\calT$ and $\calT_\calC$.

\begin{thm}
    \label{t.positiveasymptotic}
    Assume Hypotheses \ref{h.main} and \ref{h.additional}, that $\Omega$ is connected and that the conditions of 
    Theorem \ref{t.characterization}{\rm (a)}
    are satisfied so that the semigroups $\calT$, $\calT_\calC$ and $T_C$ are positive.
    \begin{enumerate}[\upshape (a)]
        \item If the conditions of Theorem \ref{t.characterization}{\rm (c)} are satisfied, then $\spb(-A_C)=0$ and there exists a strictly positive
        measure $\rho$ on $\overline{\Omega}$ and constants $M,\omega>0$ such that
        \[
        \| T_C(t) - \langle \cdot, \rho\rangle \one \|_{\calL(C(\overline{\Omega}))} \leq M e^{-\omega t}\qquad (t\geq 0).
        \]
        \item If the conditions of Theorem \ref{t.characterization}{\upshape(b)} are satisfied but those of Theorem \ref{t.characterization}{\upshape(c)} 
        are not satisfied, then $\spb(-A_C)<0$ and there exist constants $M,\omega>0$ such that
        \[
        \| T_C(t)\|_{\calL(C(\overline{\Omega}))} \leq M e^{-\omega t}\qquad (t\geq 0).
        \]
        \item If either $\Div c + B_{11}\one_\Omega +B_{12}\one_\Gamma\not\leq 0$ or
        $B_{21}\one_\Omega + B_{22}\one_\Gamma \not\leq  c\cdot \nu$, then $\spb(-A_C)>0$ and there exist constants $M, \omega >0$ such that
        \[
        \|T_C(t)\|_{\calL(C(\overline{\Omega}))} \geq Me^{\omega t}\qquad (t\geq 0).
        \]
        \end{enumerate}
\end{thm}

\begin{proof}
    (a) Since $T_C$ is Markovian it is $\omega_0(T_C)=0$, and thus
    $\spb(-A_C)=\spb(-\calA)=0$ by Proposition \ref{p.spectrum}. 
    Since $\calT$ is irreducible by Proposition \ref{p.irreducible}, it follows
    from \cite[Proposition C-III.3.5]{schreibmaschiene} that $0$ is a first order pole of the resolvent of $\calA$ and
    the corresponding eigenspace is one-dimensional, hence spanned by $\one$. By \cite[Proposition A.5]{AKK18}
    also $T_C$ is irreducible and the same results follow for 
    $A_C$. The result about asymptotic behavior now   
    follows from \cite[Theorem C-IV.2.1]{schreibmaschiene}, see also \cite[Theorem A.2]{AKK18}.\medskip

    (b) Under the assumptions of the theorem, the proof of Theorem \ref{t.characterization} yields a Markovian semigroup $\tilde{\calT}$
    such that $0\leq \calT(t)\leq \tilde{\calT}(t)$ for all $t\geq 0$. By part (a), for the generator $-\tilde{A}_C$ of $\tilde{T}_C$, 
    we have $\spb(-\tilde{A}_C) = 0$. 
    As the conditions of Theorem \ref{t.characterization}(c)
    are not satisfied, the generator $-A_C$ of $T_C$ is different from $-\tilde{A}_C$ and Theorem \ref{t.monotone} yields
    $\spb(-A_C)< 0$. Since the growth bound and spectral bound of $T_C$ coincide, the claim follows.\medskip 

    (c) Again, denote by $\frh$ the form $\fra$ with $\calB\equiv 0$ and by $\calL$ and $\calS$ the associated operator and semigroup. 
    It follows from Theorem \ref{t.characterization} with $\calB\equiv 0$
    that $\calS$ is Markovian, whence $\spb(-\calL) = 0$ by part (a) and Proposition \ref{p.spectrum}.
    Using that $B_{11}$ and $B_{22}$ satisfy the positive minimum principle and that $B_{12}, B_{21}$ are positive, it follows that
    $\frh [\calu, \calv] \leq \fra[\calu, \calv]$ for all $0\leq \calu, \calv \in D(\fra) = D(\frh)$. 
    Thus, using the Ouhabaz criterion
    for domination \cite[Theorem 2.2.4]{Ouhabaz}, it follows that  $0\leq \calS\leq \calT$. By Proposition \ref{p.irreducible} $\calT$
    is irreducible and by Theorem \ref{t.generate} it is compact. However, by our assumption, $\calA\one\neq 0$, so that
    $\calL\neq\calA$. Thus, Theorem \ref{t.monotone} implies $0=\spb(-\calL) < \spb(-\calA) =\colon \omega$. As $\calT$ is irreducible,
    we find a strictly positive function $\calu$ such that $\calT(t)\calu = e^{\omega t}\calu$ for all $t\ge 0$, see \cite[Proposition C-III.3.5]{schreibmaschiene}.
    The claim follows.    
\end{proof}

\subsection{Perturbations of dissipative, self-adjoint operators}

In this section, we make additional assumptions on the coefficients $b$ and $c$ appearing in the form $\frq$. 

\begin{hyp}
    \label{h.sym}
    Assume that $b=c \in W^{1,\infty}(\Omega; \R^d)$ satisfy $\Div b = \Div c = 0$ and  $b \cdot \nu =c \cdot \nu=0$.
\end{hyp}

As before, $\calL$ denotes the operator associated to the form $\frh$ (i.e. $\fra$ with $\calB=0$), and $\calS$ the semigroup generated by $-\calL$.

\begin{lem}
    \label{l.diss}
    Assume that Hypothesis \ref{h.main} is satisfied and $b,c \in W^{1,\infty}(\Omega; \R^d)$. Then the following are equivalent:
    \begin{enumerate}[\upshape (i)]
        \item $-\calL$ is self-adjoint and $\calS$ is Markovian.
        \item Hypothesis \ref{h.sym} is satisfied.
    \end{enumerate}
    In that case $-\calL$ is dissipative with $\spb(-\calL) = 0$.
\end{lem}
\begin{proof}
    $\calL$ is self-adjoint if and only if $\frq$ is symmetric, which in turn is equivalent to $b=c$. It follows from Theorem \ref{t.characterization}(c), that $\calS$ is Markovian if and only if $\Div c=0$ and  $c \cdot \nu=0$. This shows the equivalence of (i) and (ii).
    
    Moreover, for
    $\calu \in D(\fra)$, Hypothesis \ref{h.sym} enforces
    \begin{align*}
        \frh[\calu, \calu] & =\frq[u_1,u_1]= \int_\Omega \sum_{i,j=1}^d a_{ij}D_i u_1 \overline{D_j u_1} + \sum_{j=1}^d b_j (D_j u_1)\overline{u_1} 
        + \sum_{j=1}^d b_j u_1 \overline{D_j u_1}\dd \lambda\\
        & = \int_\Omega \sum_{i,j=1}^d a_{ij}D_i u_1 \overline{D_j u_1}  + \int_\Omega \sum_{j=1}^d b_j D_j |u_1|^2\dd \lambda \\ 
        &= \int_\Omega \sum_{i,j=1}^d a_{ij}D_i u_1 \overline{D_j u_1}- \int_\Omega  \Div b |u_1|^2 \dd \lambda +\int_\Gamma \nu \cdot b |u_1|^2 \dd \sigma \\
        &= \int_\Omega \sum_{i,j=1}^d a_{ij}D_i u_1 \overline{D_j u_1}\geq \eta \int_\Omega |\nabla u_1|^2 \dd \lambda \geq 0.
    \end{align*}

    Thus $\frh$ is accretive, or equivalently $-\calL$ is dissipative. In particular, $\spb(-\calL) \leq 0$. Equality is ensured by $\calL \one=0$.
\end{proof}

\begin{thm}
    \label{t.perturbspectral}
    Assume that Hypotheses \ref{h.main} and \ref{h.sym} are satisfied, and that the operator $\calB$ is dissipative. Then $\spb(-\calA) \leq 0$ and $\sigma(-\calA)\cap i\R \subset \{0\}$.
    Moreover, if $0 \in \sigma(-\calA)$, then $\ker (-\calA) = \linSpan(\one)$. In this case, $\spb(-\calA) = 0$, and $0$ is a dominant and algebraically simple eigenvalue.
\end{thm}

\begin{proof}
    For all $\calu\in D(\fra)$, we have
    \[
    \Re \fra[\calu, \calu] = \frq [u_1, u_1] - \Re \langle \calB\calu, \calu \rangle_\calH \geq 0
    \]
    by Lemma \ref{l.diss} and the assumption that $\calB$ is dissipative. It follows that $-\calA$ is dissipative
    and hence $\calT$ is contractive. This implies that $\spb(-\calA) = \omega_0(\calT) \leq 0$ (recall Proposition~\ref{p.spectrum}(c) for the equality of spectral and growth bounds).

    Now assume that $i\omega \in \sigma(-\calA)\cap i\R$ for some $\omega\in \R$. As $\calA$ has compact resolvent,
    $i\omega$ is an eigenvalue and hence we find $\calv \in \calH$ with $\|\calv\|_\calH=1$ and $-\calA \calv = i\omega\calv$. It follows that $\fra [\calv,\calv] = \langle -\calA\calv,\calv \rangle_\calH = i\omega$
    and hence
    \[
    \frq[v_1,v_1] = i\omega + \langle \calB\calv, \calv\rangle_\calH. 
    \]
    Taking real parts and using that Hypothesis \ref{h.sym} also yields the accretivity of $\frq$ itself, the dissipativity of $\calB$ shows  
    \[
    0\leq \eta \int_\Omega |\nabla v_1|^2\, d\lambda \leq \frq[v_1,v_1]= \Re \langle \calB\calv, \calv\rangle_\calH \leq 0.
    \]
  Therefore $\nabla v_1 = 0$, which shows that $v_1$ (and also $\calv$, as $\calv \in D(\fra)$) is necessarily constant. It follows from Hypothesis \ref{h.sym} that $\calL\calv=\alpha\cdot \calL \one=0$ for some $\alpha \in \C$ and we thus find $i\omega= -\langle \calB \calv, \calv\rangle_\calH = -|\alpha|^2 \langle \calB\one, \one\rangle_\calH \in \R$. This implies that $\omega=0$. Moreover,  we see that $0$ is dominant and geometrically simple.

  To prove the last assertion of the theorem, we require some facts about the adjoint semigroup. Recall that the adjoint generator $-\calA^*$ and dual semigroup $\calT^*$ arise from the \emph{adjoint form}
    \begin{equation*}
        \fra^*[\calu, \calv] \coloneqq \overline{\fra[\calv, \calu]}, \qquad \calu, \calv\in D(\fra).
    \end{equation*}
    Since $\calB$ is dissipative if and only if $\calB^*$ is dissipative, Lemma~\ref{l.diss} again shows that
    \begin{equation*}
        \Re\fra^*[\calu,\calu] = \frq[u_1,u_1] - \Re\langle\calB^* \calu,\calu\rangle_\calH \ge 0
    \end{equation*}
    for all $\calu\in D(\fra)$. As above, we deduce that $\spb(-\calA^*) \le 0$ and $\sigma(-\calA^*)\cap i\R \subset \{0\}$. Due to the relation
    \begin{equation*}
        \sigma(-\calA^*) = \sigma(-\calA)^* \coloneqq \{ \overline{\mu} \suchthat \mu\in\sigma(-\calA) \},
    \end{equation*}
    see e.g.\ \cite[Theorem III.6.22]{Kato76}, it follows that $0\in\sigma(-\calA^*)$ if and only if $0\in\sigma(-\calA)$. Thus if the latter holds, then $0$ is also an eigenvalue of $-\calA^*$, and the previous computations can be applied to $-\calA^*$ to show that the $0$-eigenspace of $-\calA^*$ is one-dimensional and spanned by $\one$.

    Finally, to show that $0$ is an algebraically simple eigenvalue of $-\calA$, it suffices to show that $\ker((-\calA)^2) \subseteq \ker(-\calA)$, since $0$ is a pole of $R(\argument,-\calA)$ by Proposition~\ref{p.spectrum}(a). Hence let $0 \ne \calu\in\ker((-\calA)^2)$, so that $-\calA\calu \in \ker(-\calA)$. Since $\ker(-\calA) = \linSpan(\one)$, there exists $\alpha\in\C\setminus\{0\}$ such that $\alpha \calA\calu \ge 0$. Now observe that
    \begin{equation*}
        \langle\one,\alpha\calA\calu\rangle_\calH = \langle \calA^*\one,\alpha\calu\rangle_\calH = 0,
    \end{equation*}
    because $\calA^* \one = 0$. This implies $\alpha\calA\calu = 0$, and consequently $\calu\in\ker(-\calA)$ as required.
\end{proof}

\begin{cor}
    \label{c.sais0}
    In the situation of Theorem \ref{t.perturbspectral}, one has $\spb(-\calA)=0$ if and only if $\calB\one=0$. Thus,
    if $\calB\one\neq 0$, then $\spb(-\calA) < 0$.
\end{cor}

\begin{proof}
    If $\spb(-\calA)=0$, then Theorem \ref{t.perturbspectral} yields $-\calA \one=0$ and hence
    $\fra[\one, \calu]=0$ for all $\calu\in D(\fra)$. Noting that $D_j\one=0$ for all $j=1, \ldots, d$, it follows that
    \begin{align*}
        \fra[\one, \calu] & = \int_\Omega \sum_{j=1}^d b_j\one \overline{D_j u_1}\dd \lambda -\langle \calB\one, \calu\rangle_\calH\\
        & = -\int_\Omega (\Div b)\overline{u_1}\dd \lambda +\int_\Gamma (b \cdot \nu) \overline{u_1} \dd \sigma  -\langle \calB\one, \calu\rangle_\calH = -\langle \calB\one, \calu\rangle_\calH.
    \end{align*}
    Thus, we find $\langle\calB\one, \calu\rangle_\calH= 0$ for all $\calu\in D(\fra)$ which, by density of $D(\fra)$ in $\calH$, implies
    $\calB\one=0$.

    Conversely, if $\calB\one=0$, then $\one\in \ker(-\calA)$, whence $0\in \sigma(-\calA)$. At this point, Theorem 
    \ref{t.perturbspectral} yields $\spb(-\calA)=0$.
\end{proof}

\begin{example}
    \label{ex.skewselfadjoint}
    A particular example of a dissipative operator $\calB$ is a \emph{skew-symmetric}
    (or \emph{skew-adjoint}) operator, i.e.\ $\calB^* = -\calB$.
    In this case, we have $\Re\langle \calB\calu, \calu\rangle_\calH = 0$ for all $u\in \calH$. We note that $\calB$ is skew-symmetric
    if and only if both $B_{11}$ and $B_{22}$ are skew-symmetric and $B_{12}^*= - B_{21}$.

    Let us give a particular example for this in the case of integral operators, see Example \ref{ex.integraloperator}.
    We choose $B_{11}=0$ and $B_{22}=0$ and let $k\in L^\infty(\Omega\times\Gamma; \R)$. Define
    \[
    [B_{12}u_2](x) = \int_\Gamma k(x, z)u_2(z)\dd\sigma(z)\quad\mbox{and}\quad [B_{21}u_1](z) = - \int_\Omega k(x,z)u_1(x)\dd\lambda(x).
    \]
    Then, the resulting operator $\calB$ is skew-symmetric. We note that in this case, we have $\calB\one=0$ if and only if
    \[
    \int_\Omega k(x,z)\dd\sigma(z) = 0 \mbox{ for } \lambda\mbox{-almost all } x\in \Omega
    \]
    and
    \[
    \quad \int_\Gamma k(x,z)\dd\lambda (x) = 0 \mbox{ for } \sigma\mbox{-almost all } z\in \Gamma.
    \]
    These conditions are satisfied, for example, if $k(x,z) = f(x)g(z)$ where $f\in L^\infty(\Omega)$ and $g\in L^\infty(\Gamma)$
    satisfy $\int_\Omega f\dd\lambda = \int_\Gamma g\dd\sigma=0$.
\end{example}

In the situation of Theorem \ref{t.perturbspectral}, the operator $-\calA = - \calL + \calB$ is the sum of two dissipative operators
and thus dissipative itself. We close this section by giving an example showing that merely assuming that $-\calA$ is dissipative
is not sufficient to obtain the conclusions of that theorem.

\begin{example}
    \label{ex.dissipative}
    Let $\Omega = (0,1) \subset \R$ so that $\Gamma = \{0,1\}$. In this case $L^2(\Gamma)\simeq \C^2$. For $\calu$ in $D(\fra)$,
    we have $u_1\in H^1(0,1)\subset C([0,1])$ and we may (and shall) identify $u_2=\trace u_1$ with the vector
    $(u_1(0), u_1(1)) \in \C^2$.
    
    For our example, we choose $a_{11}=1$, $b_1=c_1=0$
    as well as $B_{11}$, $B_{12}$ and $B_{21}$ the appropriate 0 operators. Finally, let 
    \[
    B_{22} = \begin{pmatrix}
        1 & -1\\
        -1 & 1
    \end{pmatrix}.
    \]
    Then $B_{22}$ is symmetric and $\sigma(B_{22}) = \{0,2\}$ so that $B_{22}$ (and hence $\calB$) is \emph{not} dissipative.
    It is easy to see that
    \[
    \langle B_{22}u_2, u_2\rangle_\Gamma = |u_1(1)-u_1(0)|^2.
    \]
    This implies that
    \begin{align*}
    \fra[\calu, \calu] & = \int_0^1|u_1'(t)|^2\dd t - |u_1(1)-u_1(0)|^2 \\
    & = \int_0^1|u_1'(t)|^2\dd t - \Big|\int_0^1 u_1'(t)\dd t\Big|^2 \geq 0,
    \end{align*}
    by Jensen's inequality. Thus, $\fra$ is accretive, implying that $-\calA$ is dissipative, whence $\spb(-\calA) \leq 0$. We
    will show that $0\in \sigma(-\calA)$, so that actually $\spb(-\calA) =0$, but $\ker(-\calA)$ is two-dimensional.

    Indeed, for $\calu\in D(-\calA)$, it follows from Proposition \ref{p.domain} that $-\calA\calu = 0$ if and only if $u_1''=0$ and $\partial_\nu u_1=B_{22} u_2$. The boundary condition translates to
    \begin{align*}
        - u_1'(0) & = u_1(0) - u_1(1)\\
        u_1'(1) & = u_1(1) - u_1(0).
    \end{align*}
    Note that $u_1''=0$ implies that $u_1(t) = a+bt$. But it is easy to see that the boundary conditions are satisfied 
    \emph{independently} of the choice of $a$ and $b$. This shows that $0\in \sigma(-\calA)$ and that
    $\dim\ker(-\calA)=2$.
\end{example}

\section{Eventual positivity}

In Section \ref{subsec.positive-semigroup}, we studied the asymptotic behavior of the semigroup $\calT$ under conditions which ensured the positivity of the semigroup. The advantage in this setting is that we could draw on well-established results in the spectral theory of positive semigroups. However, if the semigroup is not positive, one could ask about positivity for sufficiently large times. This leads to the question of \emph{eventually positive} solutions to evolution equations. While isolated examples of such behavior were known for several decades for matrix semigroups and in the PDE literature, a systematic theory of eventually positive semigroups on infinite dimensional Banach lattices was initiated fairly recently in the papers~\cite{DGK2, DGK1, DG18}. This topic has rapidly developed in the last few years, and the interested reader may consult the recent survey article~\cite{glueck2022} for an overview of the current state of the theory.
Let $(T(t))_{t\ge 0}$ be a strongly continuous semigroup on the Banach lattice $E$. It is natural to call the semigroup $(T(t))_{t\ge 0}$ \emph{eventually positive} if for every $f\ge 0$, there exists $t_0 = t_0(f) \ge 0$ such that
\begin{equation}
\label{eq.e-pos-simple}
    T(t)f \ge 0 \qquad\text{for all }t\ge t_0.
\end{equation}
Many of the general results currently known about eventually positive semigroups are inspired by classical Perron--Frobenius theory and the spectral theory of positive semigroups. For example, it is shown in~\cite[Theorem 7.6]{DGK1} that if $A:D(A)\subset E\to E$ generates a strongly continuous semigroup that satisfies~\eqref{eq.e-pos-simple} and $\sigma(A)\ne\emptyset$, then the spectral bound $\spb(A)$ is a spectral value.

However, the general theory is more fruitful if we consider a stronger notion of eventual positivity, which we will now introduce in the specific context of $L^p$-spaces.
\begin{defn}
\label{def.e-pos-strong}
    Let $(M,\mu)$ be a finite measure space, and let $T = (T(t))_{t\ge 0}$ be a strongly continuous semigroup on the Banach lattice $E = L^p(M,\mu)$. We say that $T$ is \emph{eventually strongly positive} if for every $f\in E_+\setminus\{0\}$, there exists a constant $\delta=\delta_f >0$ and $t_0 = t_0(f) \ge 0$ such that
    \begin{equation*}
        T(t)f \ge \delta\one \qquad\text{for all }t\ge t_0.
    \end{equation*}
    If the time $t_0$ can be chosen independently of $f\in E_+$, then we say that $T$ is \emph{uniformly eventually strongly positive}.
    Note that in this case, for every $t\geq t_0$ the operator $T(t)$ is strictly positive, as $T(t)f\gg 0$ for every $f>0$.
\end{defn}

\begin{rem}
\label{r.terminology}
    In accordance with \cite{DGK1,DGK2}, it would be more appropriate to refer to the notion of (uniform) eventual strong positivity in 
    Definition~\ref{def.e-pos-strong}  as \emph{individual} (respectively \emph{uniform}) \emph{eventual strong positivity with respect to
    the quasi-interior point $\one$}. 
    The general theory developed in \cite{DGK1, DGK2} allows for arbitrary quasi-interior points $u\in E_+$ instead of $\one$.
\end{rem}

In practice, the notion of eventual strong positivity in Definition~\ref{def.e-pos-strong} is related to the question of asymptotic behavior and lower bounds on solutions of evolution equations. For our applications, we are particularly interested in the eigenspace corresponding to the spectral bound $\spb(A)$, and hence we introduce the following notion: an operator $P:L^p(M, \mu) \to L^p(M, \mu)$ is called \emph{strongly positive} if for all $f>0$, there exists $\delta = \delta(f)>0$ such that
\begin{equation}
\label{eq.strongly-pos-proj}
    Pf \ge \delta \one.
\end{equation}
Another key ingredient for our purposes is the \emph{smoothing condition}
\begin{equation}
\label{eq.smoothing-Eu}
    T(t_1)L^p(M,\mu) \subset L^\infty(M,\mu) \quad\text{for some }t_1 > 0.
\end{equation}
By combining this condition with spectral information about the generator, the following characterization of eventually strongly positive semigroups can be given. For simplicity, we only state the result for generators with compact resolvent, which is the case considered in this article.

\begin{thm}
\label{t.eventual-pos-characterize}
    Let $T=(e^{tA})_{t\ge 0}$ be a real, strongly continuous semigroup on $E = L^p(M,\mu)$, such that the generator $A:D(A)\subset E\to E$ has compact resolvent. If $T$ satisfies the smoothing condition~\eqref{eq.smoothing-Eu}, then the following are equivalent:
    \begin{enumerate}[\upshape (i)]
        \item $T$ is eventually strongly positive.
        \item $\spb(A)$ is a dominant eigenvalue, and the corresponding spectral projection $P$ is strongly positive.
        \item $\spb(A)$ is a dominant eigenvalue and geometrically simple, and the corresponding eigenspace is spanned by a vector $v$ such that $v\ge \delta\one$ for some constant $\delta>0$. Moreover, the dual eigenspace $\ker(\spb(A)I - A')$ contains a strictly positive functional $\psi$ {\upshape(}i.e.\ a positive functional such that $\langle \psi,f \rangle > 0$ for all $f\in E_+\setminus\{0\}${\upshape)}.
    \end{enumerate}
    If any of the above conditions hold, then $\spb(A)$ is even an algebraically simple eigenvalue of $A$.
\end{thm}

\begin{proof}
    The equivalence of (ii) and (iii), and the fact that these conditions imply algebraic simplicity of $\spb(A)$, is a general property of strongly positive projections, which was proved in~\cite[Corollary 3.3]{DGK2}.

    The equivalence of (i) and (ii) is proved in~\cite[Theorem 5.2]{DGK2}.
\end{proof}

\begin{rem}
    For the interested reader, we point out that the implication (i) $\Rightarrow$ (ii) in Theorem~\ref{t.eventual-pos-characterize} holds even without the smoothing condition, and this was proved in~\cite[Theorem 5.1]{DG17}. The reverse implication is possible thanks to the smoothing condition, and does not hold in general, even for semigroups with bounded generators. A counterexample is shown in~\cite[Example 5.4]{DGK2}.
    
\end{rem}

We now return to the Wentzell--Robin semigroup $\calT$. Our spectral analysis in the previous section leads to a simple sufficient criterion for eventual strong positivity.
\begin{thm}
\label{t.unif-eventual-pos}
    Assume Hypotheses~\ref{h.main} and~\ref{h.sym}. Suppose that $\calB$ satisfies the following conditions:
    \begin{enumerate}[\upshape (i)]
        \item $\calB$ is dissipative and $\calB\one = 0$;
        \item $\calB$ is bounded on $L^\infty(\Omega)\times L^\infty(\Gamma)$ and extrapolates to a bounded operator on $L^1(\Omega)\times L^1(\Gamma)$.
    \end{enumerate}
    Then $\calT$ is eventually strongly positive.
\end{thm}

\begin{proof}
    We verify condition (iii) of Theorem \ref{t.eventual-pos-characterize}.

    From Theorem~\ref{t.perturbspectral}, we know that $\spb(-\calA)=0$, $\sigma(-\calA) \cap i\R = \{0\}$, and the associated eigenspace is one-dimensional and spanned by $\one$. In particular, $\spb(A)$ is a dominant eigenvalue of $-\calA$. Theorem~\ref{t.hoelder} shows that
    \begin{equation*}
        \calT(t)\calH \subset L^\infty(\Omega)
        \times L^\infty(\Gamma)
    \end{equation*}
    for all $t>0$, and hence $\calT$ satisfies the smoothing condition~\eqref{eq.smoothing-Eu}.

    Recall that the adjoint generator $-\calA^*$ and dual semigroup $\calT^*$ arise from the adjoint form
    \begin{equation*}
        \fra^*[\calu, \calv] \coloneqq \overline{\fra[\calv, \calu]}, \qquad \calu, \calv\in D(\fra).
    \end{equation*}
    However, since $\calA$ is real, one can show that the Hilbert space adjoints
    $\calA^*$ and $\calT^*$ coincide with the Banach space adjoints $\calA'$ and $\calT'$ --- see~\cite[p.\ 10]{DG18} for a detailed explanation. 
    In particular, if $\calB$ is bounded on $L^\infty(\Omega)\times L^\infty(\Gamma)$ and extrapolates to a bounded operator on $L^1(\Omega)\times L^1(\Gamma)$, then it makes sense to say that $\calB^*$, which is \emph{a priori} bounded on $L^1(\Omega)\times L^1(\Gamma)$, extends to a bounded operator on $(L^1(\Omega)\times L^1(\Gamma))' = L^\infty(\Omega)\times L^\infty(\Gamma)$.

    Regarding the spectrum of $-\calA^*$, we use again the relation
    \begin{equation*}
        \sigma(-\calA^*) = \sigma(-\calA)^* \coloneqq \{ \overline{\mu} \suchthat \mu\in\sigma(-\calA) \}
    \end{equation*}
    as in the proof of Theorem~\ref{t.perturbspectral}. This in particular implies that $\spb(-\calA^*) = 0$ is a dominant eigenvalue of $-\calA^*$, and Theorem~\ref{t.perturbspectral} applied to $-\calA^*$ shows that the $0$-eigenspace is one-dimensional and spanned by $\one$. Thus the dual eigenspace $\ker(-\calA^*)$ contains the strictly positive functional $\psi = \langle \one,\argument\rangle_\calH$, and Theorem \ref{t.eventual-pos-characterize} yields the claim.
\end{proof}

\begin{rem}
\label{r.e-pos}
    (i) The assumption (ii) in Theorem~\ref{t.unif-eventual-pos} is not optimal. 
    Indeed, recalling the conditions of Theorem~\ref{t.hoelder}, we can omit condition (ii) if $d=1$ and $d=2$.
    In case $d\geq 3$, we may replace (ii) with the following more general but technical assumption: 
    there exists some $p \in (d-1,\infty)$ such that $\calB$ is bounded on $L^p(\Omega)\times L^p(\Gamma)$ and 
    extrapolates to a bounded operator on $L^{p'}(\Omega)\times L^{p'}(\Gamma)$, where $p' > 1$ is the conjugate H\"older exponent.

    (ii) The semigroup $\calT$ in Theorem~\ref{t.unif-eventual-pos} is even \emph{uniformly} eventually strongly positive. This is because we can apply Theorem~\ref{t.hoelder} to the adjoint semigroup and obtain
    \begin{equation*}
        \calT^*(t)\calH \subset L^\infty(\Omega)\times L^\infty(\Gamma)
    \end{equation*}
    for all $t>0$. Combined with the spectral information on $-\calA^*$, we can then use~\cite[Theorem 3.1]{DG18} to deduce the following conclusion: there exist $t_0 \ge 0$ and $\delta > 0$ such that
    \begin{equation*}
        \calT(t)\calu \ge \delta\langle \one,\calu\rangle_{\calH}\one
    \end{equation*}
    for all $t\ge t_0$ and all $0\le\calu\in\calH$.
\end{rem}

Following Example~\ref{ex.skewselfadjoint}, we can identify a class of operators $\calB$ for which the corresponding semigroup is (uniformly) eventually strongly positive, but not positive.
\begin{example}
    In Example~\ref{ex.skewselfadjoint}, we constructed a skew-symmetric (hence dissipative) operator $\calB$ that satisfies $\calB\one = 0$ via a real-valued kernel function $k\in L^\infty(\Omega\times \Gamma ;\R)$. 
    Such an operator $\calB$ is clearly bounded on $L^\infty(\Omega)\times L^\infty(\Gamma)$, and also extrapolates to a bounded linear operator on $L^1(\Omega)\times L^1(\Gamma)$. Hence $\calB$ satisfies the assumptions of Theorem~\ref{t.unif-eventual-pos}. 
    By Remark~\ref{r.e-pos}(ii) we know that the operator $-\calA$ associated to such $\calB$ generates a uniformly eventually strongly positive semigroup $\calT$.

    However, if $k$ is not equal to $0$ almost everywhere, then the conditions
    \begin{equation*}
        \left\{ \begin{aligned}
        \int_\Omega k(x,z)\dd\sigma(z) &= 0 \quad\text{for } \lambda\text{-a.e. } x\in\Omega \\
        \int_\Gamma k(x,z)\dd\lambda(x) &=0 \quad\text{for } \sigma\text{-a.e. } z\in\Gamma
        \end{aligned}\right.
    \end{equation*}
    imply that $k$ changes sign in $\Omega\times \Gamma$ so that $B_{12}$ and $B_{21}$ from Example~\ref{ex.skewselfadjoint} are not positive operators. Consequently, by the characterization in Theorem~\ref{t.characterization}(a), we deduce that the semigroup $\calT$ is not positive.
\end{example}

\section{A 1-dimensional example in detail} \label{sect.perturbation}
In this final section, we examine in detail a one-dimensional example that illustrates the variety of effects that can occur when we add a very simple non-local, skew-symmetric perturbation to an operator that generates a positive semigroup. 
To that end, we investigate a slightly different $B_{22}$ than in Example~\ref{ex.dissipative} and consider the following situation.

\begin{hyp}\label{hyp.perturbation}   
 Let $\Omega = (0,1) \subset \R$, $\Gamma = \{0,1\}$. Let $a_{11}=1$, $b_1=c_1=0$, and let $B_{11}$, $B_{12}$ and $B_{21}$ be the appropriate 0 operators. Finally, let 
    \[
    B_{22} = \begin{pmatrix}
        0 & 1\\
        -1 & 0
    \end{pmatrix}
    \]
    and consider the family of real operators $-\calA_\tau=-\calL+\tau\calB$ for $\tau \in \R$.
\end{hyp} 
It is easily seen, that Hypothesis \ref{hyp.perturbation} automatically implies Hypotheses \ref{h.main} and \ref{h.sym}. This example illustrates the behavior of perturbing a positive operator with a small skew-adjoint matrix on the boundary. Slowly increasing the perturbation parameter $\tau$, we observe that positivity is lost instantly, but eventual positivity is maintained in a certain parameter range. Increasing the perturbation parameter further, we see that eventual positivity will fail for different reasons as one by one the necessary conditions from Theorem~\ref{t.eventual-pos-characterize}~(iii) cease to be fulfilled. More precisely we have the following behavior.

\begin{thm}\label{t.ev.pos}
Assume Hypothesis \ref{hyp.perturbation}. Then there are values $0<\tau_p<\tau_s<\tau^*$ {\upshape(}defined respectively in Formulae \eqref{taup}, \eqref{taus}, and \eqref{taustar} below{\upshape)} such that for $|\tau| <\tau^*$ the following behavior occurs:
\begin{enumerate}[\upshape (a)]
    \item The semigroup $\calT_\tau$ is positive if and only if $\tau=0$.
    \item The semigroup $\calT_\tau$ is eventually strongly positive in the sense of Definition~\ref{def.e-pos-strong}
    if and only if $|\tau|<\tau_p$.
    \item If $|\tau|\in[\tau_p,\tau^*)$ the semigroup $\calT_\tau$ is not eventually strongly positive.
    More precisely:
    \begin{enumerate}[\upshape(i)] 
    \item If $|\tau|=\tau_p$, the spectral bound $\spb(-\calA_\tau)$ is a dominant, algebraically simple eigenvalue, whose eigenspace is spanned by a positive {\rm (}but not strictly positive{\rm )} function. 
    \item If $|\tau| \in (\tau_p,\tau_s)$, the spectral bound $\spb(-\calA_\tau)$ is a dominant, algebraically simple eigenvalue whose eigenspace is spanned by a function with sign change. 
    \item If $|\tau| =\tau_s$, the spectral bound $\spb(-\calA_\tau)$ is a dominant, geometrically simple eigenvalue that is not algebraically simple as the resolvent has a pole of order two.
    \item If $|\tau| \in (\tau_s, \tau^*)$, the spectral bound $\spb(-\calA_\tau)$ is not contained in the spectrum. Instead, there is pair of complex conjugate eigenvalues $\lambda \in \C$ with $\Re \lambda=\spb(-\calA_\tau)$.
\end{enumerate}
\end{enumerate}
\end{thm}

\begin{rem} (i) In the case $|\tau|<\tau_p$, we even have uniform eventual strong positivity, see Remark~\ref{r.e-pos}(ii).

(ii) In the case $|\tau|=\tau_p$, it follows from \cite[Theorem 8.3]{DGK2} that the semigroup $\calT_\tau$ is at least
\emph{asymptotically positive} in the sense that 
    \[
    \mathrm{dist}(e^{-t \spb(-\calA_\tau)}\calT_\tau(t) \calf, \calH_+) \to 0\qquad\mbox{ as } t\to \infty
    \]
for every $\calf\in \calH_+$. Note that the rescaled semigroup $(e^{-t\spb(-\calA_\tau)}\calT_\tau(t))_{t\geq 0}$ has growth bound and spectral bound
$0$. Thus, asymptotic positivity means that for positive initial data, the orbit under the semigroup, when appropriately rescaled, approaches
the positive cone $\calH_+$ as $t\to \infty$. 
\end{rem} 

Before we prove Theorem \ref{t.ev.pos}, we need some preparation. Firstly, we collect general spectral properties of $\calA_\tau$. Note that $(0,\infty) \in \rho(-\calA_\tau)$, so if $\lambda \in \sigma(-\calA_\tau)$, then $-\lambda \in \C\setminus (-\infty, 0]$. We let
$\sqrt{\argument} : \C\setminus (-\infty, 0]\to \C$ denote the principal branch of the square root. For $-\lambda \in \C\setminus (-\infty, 0]$
we set $\mu = \sqrt{-\lambda}$ and $w = i\mu = i\sqrt{-\lambda}$. Note that with this convention we always have $\Re\mu\geq 0$ and $\Im w \geq 0$.

\begin{prop}\label{prop.spect.prop}
   Assume Hypothesis \ref{hyp.perturbation}. Then it is $\calA_\tau^*=\calA_{-\tau}$ and $\sigma(\calA_\tau)=\sigma(\calA_{-\tau})$. 
   Moreover, $\lambda \in \sigma(-\calA_\tau)$ if and only if $\mu=\sqrt{-\lambda}$ satisfies
       \begin{equation}\label{eq.mu}
       \cot(\mu)=\frac{\tau^2-\mu^2+\mu^4}{2\mu^3}.
       \end{equation}     
    All spectral values of $-\calA_\tau$ are isolated eigenvalues which are geometrically simple.
    For all $\tau \neq 0$, it holds that $\spb(-\calA_\tau) <0$.
\end{prop}

\begin{proof}
By Proposition \ref{p.spectrum}(a), all spectral values are isolated eigenvalues with corresponding finite-dimensional eigenspaces.
We see directly that $\tau\calB$ is skew-symmetric, but for $\tau\neq 0$ we have $\tau\calB \one = \tau B_{22} \binom{1}{1}=\binom{\tau}{-\tau}\neq 0$. Thus Corollary \ref{c.sais0} shows that $\spb(-\calA_\tau)<0$. Rewriting $-\calA_\tau \calu=\lambda \calu$ yields the eigenvalue problem
\begin{equation} \label{eq:eigenvalue1D}
\begin{aligned}
        u''(x) & =\lambda u(x),\\
        u'(0)+\tau u(1)& =\lambda u(0),\\
        -u'(1)-\tau u(0) & = \lambda u(1).
\end{aligned}
\end{equation}
Note that for $\tau\neq 0$, 
complex eigenvalues $\lambda$ may also occur, so we use the complex ansatz $\alpha e^{wx}+\beta e^{-wx}$ where $\lambda=w^2$. By standard ODE theory, all solutions of $u''(x)=\lambda u(x)$ are of this shape whenever $\lambda \neq 0$, thus
in particular when $\tau\neq 0$.

A short calculation shows that the boundary condition translates to
$M_w \binom{\alpha}{\beta}=0$ with 
\[ 
M_w=\begin{pmatrix}
w-w^2+\tau e^w & -w-w^2+\tau e^{-w}\\ -w e^{w}-w^2 e^w-\tau & w e^{-w}-w^2 e^{-w}-\tau
\end{pmatrix}.
\]
Note that
\begin{equation} \label{eq:det}
\det M_w=2(-\tau^2-w^2-w^4)\sinh(w)-4 w^3 \cosh(w), 
\end{equation}
which shows that the spectrum only depends on the absolute value of $\tau$. 
Also observe that $\det M_w=0$ if and only if
\[4w^3 \cosh(w)=-2(\tau^2+w^2+w^4)\sinh(w) \text{ or } \coth(w)=-\frac{\tau^2+w^2+w^4}{2w^3}, \] which is equivalent to \eqref{eq.mu} as $\mu=-iw$.

Finally, as $L^2(\Gamma)\simeq \C^2$ is two-dimensional, the dimension of the kernel is at most two. Double eigenvalues, however, can only occur when all entries of $M_w$ are zero. But if the entries on the diagonal are zero, we must have 
$\tau e^w=(w^2-w)$ and $\tau=(w-w^2)e^{-w}$, which yields $\tau=0$. For $\tau=0$, the matrix can only vanish if $w=0$, which was excluded. 
That the unperturbed case has no double eigenvalue at $0$ follows from Theorem \ref{t.perturbspectral}. 
\end{proof}

In the ensuing investigations, we focus on values of $\tau$ close to $0$.
As $-\calA_\tau = -\calL + \tau \calB$ is a perturbation of $-\calL$, perturbation arguments show that for $|\tau|$ in a certain range
the first two eigenvalues of $-\calA_\tau$ are obtained as perturbations of the first two eigenvalues of $-\calL$. The latter can be obtained from Proposition \ref{prop.spect.prop} setting $\tau=0$:

\begin{cor}\label{c.A.nopert}
It is $\sigma(-\calL)\subset (-\infty, 0]$ and $\lambda \in \sigma (-\calL)$ if and only if
$\lambda =-\mu^2$ for a solution $\mu$ of $\cot \mu = \frac{\mu^2-1}{2\mu}$. The largest three eigenvalues are
$\lambda_1(0)=0$, $\lambda_2(0)\approx -1.707$ and $\lambda_3(0) \approx -13.492$. They satisfy $\sqrt{-\lambda_2(0)}<\frac\pi2<\pi<\sqrt{-\lambda_3(0)}.$
\end{cor}

Now set 
\begin{equation}\label{taustar}
    \tau^*=\frac12|\lambda_3(0)-\lambda_2(0)|\approx 5.891.
\end{equation} Then $\lambda_2(0)-\tau^* \approx -7.598$. Set
\[\bbH:=\{\lambda \in \C \suchthat \lambda_2(0)-\tau^* \leq \Re \lambda\}\]
and 
\[ \bbS:=\{\lambda \in \C \suchthat \lambda_2(0)-\tau^* < \Re \lambda \leq 0, |\Im \lambda|< \tau^*\} \subset \bbH.\] 

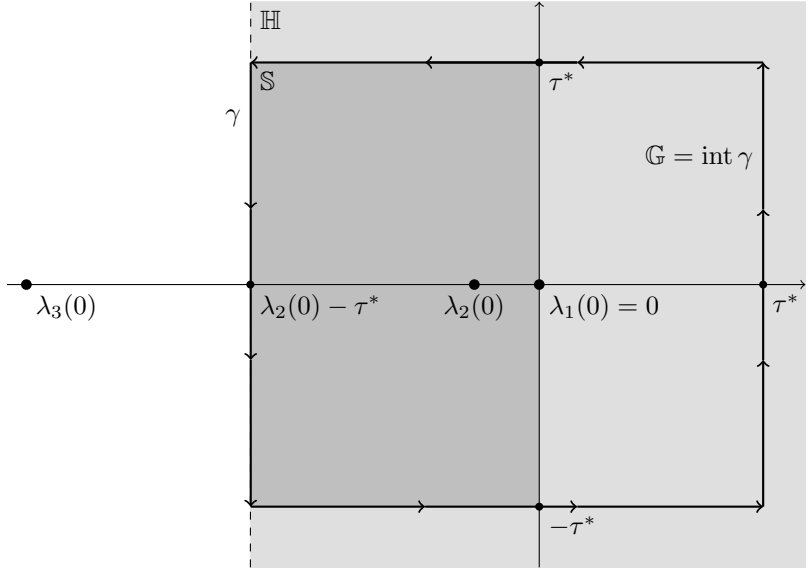
\begin{figure}[ht]
    \centering
    \begin{tikzpicture}[scale=0.5]
     
    \draw[fill=gray!25,gray!25] (-7.589,7.5) node[below right, black]{$\bbH$} -- (-7.589,-7.5)-- (7,-7.5) -- (7,7.5) -- (-7.589,7.5);
     \draw[dashed,black](-7.598,-7.5) --(-7.589, 7.5); 
    \draw[fill=gray!50,gray!50] (-7.589,5.89) node[below right,black]{$\bbS$} -- (-7.589,-5.89)-- (0,-5.89) -- (0,5.89) -- (-7.589,5.89);
    \draw[black,thick,->] (-7.589,5.89) -- node[above left,black]{$\gamma$}(-7.589,2);
    \draw[black,thick,->] (-7.589,2)  -- (-7.589,-2);
    \draw[black,thick,->] (-7.589,-2)  -- (-7.589,-5.89);
    \draw[black,thick,->] (-7.589,-5.89) -- (-3,-5.89); 
    \draw[black,thick,->] (-3,-5.89) -- (1,-5.89); 
    \draw[black,thick,->] (1,-5.89) -- (5.89,-5.89);
    \draw[black,thick,->] (5.89,-5.89)  -- (5.89,-2);
    \draw[black,thick,->] (5.89,-2) -- (5.89,2); 
    \draw[black,thick,->] (5.89,2) -- node[below left, black]{$\bbG=\mathrm{int}\,\gamma$}(5.89,5.89); 
    \draw[black,thick,->] (5.89,5.89) -- (1,5.89);
    \draw[black,thick,->] (1,5.89) -- (-3,5.89);
    \draw[black,thick,->] (1,5.89) -- (-7.589,5.89);

        \draw[->] (-14,0) -- (7,0);
		\draw[->] (0,-7.5) -- (0,7.5);

       
       \fill(-7.598,0)circle(3pt) node[below right]{$\lambda_2(0)-\tau^*$};
       \fill(0,0)circle(4pt) node[below right]{$\lambda_1(0)=0$};
       \fill(-1.707,0)circle(4pt) node[below]{$\lambda_2(0)$};
       \fill(-13.49,0)circle(4pt) node[below right]{$\lambda_3(0)$};
        \fill(0,5.89)circle(3pt) node[below right]{$\tau^*$};
        \fill(5.89,0)circle(3pt) node[below right]{$\tau^*$};
        \fill(0,-5.89)circle(3pt) node[below right]{$-\tau^*$};
   
    \end{tikzpicture}
    \caption{Spectral regions $\bbH$ and $\bbS$, and integration path $\gamma$.}
    \label{fig:enter-label}
\end{figure}

\pagebreak[3]

\begin{lem}\label{lem.spectral.projection}
Let $|\tau|<\tau^*$. Then $\sigma(-\calA_\tau)\cap \bbH\subset \bbS$ and
\begin{equation*}
     \#\big( \sigma(-\calA_\tau) \cap \bbH \big) \in \{1,2\}.
\end{equation*}
Moreover, exactly one of the following three cases occurs:
\begin{enumerate}[\upshape (A)]
    \item $\sigma(-\calA_\tau) \cap \bbH =\{\lambda_1(\tau),\lambda_2(\tau)\}\subset \R$ where $\lambda_2(\tau)<\lambda_1(\tau)=\spb(-\calA_\tau)$. Both eigenvalues are algebraically simple.
    \item $\sigma(-\calA_\tau) \cap \bbH =\{\lambda(\tau)\}\subset \R$ where $\spb(-\calA_\tau)=\lambda(\tau)$ is geometrically simple but not algebraically simple.
    \item $\sigma(-\calA_\tau) \cap \bbH =\{\lambda_1(\tau),\lambda_2(\tau)\}$, where $\lambda_2(\tau)=\overline{\lambda_1(\tau)}$ is a pair of complex conjugates with non-zero imaginary part. Both eigenvalues are algebraically simple.
\end{enumerate}
\end{lem}
\begin{proof}
We adapt the strategy from \cite[Lemma 3.3]{DG18-perturbation} to our situation. As $\sigma(-\calL)\cap \bbH = \{\lambda_1(0), \lambda_2(0)\}$, 
an easy perturbation argument based on the Neumann series shows that if $|\Im \lambda|> \tau^*$, then $\lambda \in \rho(-\calA_\tau)$. Now let $\gamma$ be the path along the boundary of the open box-shaped domain
\[\bbG=\{\lambda \in \C \suchthat \lambda_2-\tau^* < \Re \lambda < \tau^*, |\Im \lambda|< \tau^*\}.\] 
As $-\calL$ is self-adjoint $\|R(\lambda,-\calL)\|^{-1}=|\dist(\lambda,\sigma(-\calL))|$, so for any $\lambda \in \bbH \setminus \bbG \supset \gamma$ we have $\|R(\lambda,-\calL)\|^{-1}=|\dist(\lambda,\sigma(-\calL))|\geq \tau^*$ or, equivalently, $\|R(\lambda,-\calL)\| \leq (\tau^*)^{-1}$.
This implies that $\bbH \setminus \bbG \subset \rho(-\calA_\tau)$ for $\tau < \tau^*$. Indeed, 
\[R(\lambda,-\calA_\tau)=R(\lambda,-\calL)(I-\tau\calB R(\lambda,-\calL))^{-1}=R(\lambda,-\calL)\sum_{k=0}^\infty [\tau\calB R(\lambda,-\calL)]^k,\] and the latter converges absolutely, as $\|\tau\calB R(\lambda,-\calL)\|\leq |\tau|(\tau^*)^{-1}\|\calB\|<1$. 
Now consider the spectral projection
\[P_{\tau}=\frac{1}{2\pi i}\int_{\gamma} R(\lambda,-\calL+\tau\calB)^{-1} \mathrm{d}\lambda.\]
As $-\calL =-\calA_0$ has two algebraically simple eigenvalues in $\bbG$, $P_0$ has rank two. Next, we prove that $P_\tau$ depends continuously
on $\tau$ whence a perturbation result due to Kato \cite[Lemma I.4.10]{Kato76} yields that $P_\tau$ has rank two for \emph{all} $|\tau|<\tau^*$.
To that end, set
$\alpha:=\min_{\lambda \in \gamma} \|R(\lambda,-\calA_\tau)\|^{-1}$. Then for $|\tau|,|\theta|<\tau^*$, $\delta \in (0,1)$, and $|\tau-\theta|<\alpha \delta$ we have 
\[\|R(\lambda,-\calA_\tau)-R(\lambda,-\calA_\theta)\| \leq \|R(\lambda,-\calA_\tau)\|\sum_{k=1}^\infty \|(\theta-\tau)\calB R(\lambda,-\calA_\tau)\|^k=\frac{\delta}{\alpha(1-\delta)} \] 
and thus
\[\|P_{\tau}-P_{\theta}\| < \frac{(4 \tau^*+|\lambda_2(0)|)\delta}{\pi \alpha(1-\delta)} \rightarrow 0\] 
for $\delta \rightarrow 0$.\smallskip

By what was done so far, we see that for $|\tau|<\tau^*$, the operator $-\calA_\tau$ has at most two eigenvalues in $\bbH$, 
and all of them lie in $\bbG \cap \{\lambda \suchthat \Re \lambda \leq 0\}=\bbS$. As $-\calA_\tau$ is real, if $\lambda\in \sigma(-\calA_\tau)\cap(\C\setminus\R)$, then also $\overline{\lambda}\in \sigma(-\calA_\tau)$. This shows that only one of the cases
(A), (B) or (C) can occur:

If there is only one eigenvalue, i.e.\ case (B) occurs, then it has to be a pole of order two of the resolvent, as the corresponding spectral
projection has rank two. Moreover, in this case the eigenvalue has to be real, as otherwise there would be a second eigenvalue. 
If there are two eigenvalues, the same argument shows that they are both real (case (A)) or a pair of complex conjugates (case (C)). 
In all cases, Proposition \ref{prop.spect.prop} yields the geometric simplicity of the eigenvalues.
\end{proof}

\begin{lem}\label{lem.C}
Let $J=\big(0,\sqrt{-\lambda_2(0)}\,\big)$ and $|\tau|<\tau^*$. Consider the function 
\[f: \bar J \rightarrow [0,\infty), \quad \mu \mapsto f(\mu):=\mu\sqrt{2\mu\cot(\mu)+1-\mu^2}.\] 
Then  $f(\mu)>0$ on $J$, $f(\mu)=0$ on $\partial J$ and $f''(x)<0$ on $J$.
Furthermore, for any $\lambda \in (\lambda_2(0)-\tau^*,0]$, we have $\lambda \in \sigma(-\calA_\tau)$ if and only if there exists 
$\mu\in [0, \sqrt{-\lambda_2(0)}]$ with $f(\mu)=|\tau|$, such that $\lambda = -\mu^2$.
\end{lem}

\begin{proof}
As $\lim_{\mu\to 0}\mu\cot(\mu) = 1$, we have $f(0)=0$. Straightforward calculations yield $f''(\mu)<0$ and $f(\mu)>0$ on $J$. Furthermore, $f(\mu)=0$ on $\partial J$ as $2\mu\cot(\mu)+1-\mu^2=0$ is equivalent to $\cot(\mu)=\frac{\mu^2-1}{2\mu}$ which is satisfied for $\mu=\sqrt{-\lambda_2(0)}$ by definition. 

\begin{figure}[ht]
    \centering
    \begin{tikzpicture}[scale=1]
    \begin{axis}[axis lines=middle,
        xlabel={$\mu$}, ylabel={$f(\mu)$}, xmin=-0.2, xmax=1.5,ymin=-0.1, ymax=1.3, xmajorticks=false, ytick={0,0.25,0.5,0.75,1,1.25},
        anchor=origin, disabledatascaling,x=5cm,y=5cm]
    \addplot[domain = 0.001:1.3, color=black, thick, smooth]
        {x*sqrt(2*x*(cos(deg(x))/sin(deg(x))) + 1 - x^2)} -- (1.32,0);
    \end{axis}
   \fill(5*0.9307,5*1.1474)circle(2pt) node[above right]{};
   \draw[dotted] (5*0.9307,5*1.1474)--(5*0.9307,0.05);
   \node at (5*0.9307,0)[below]{$\mu_s$};
   \fill(0,0)circle(2pt);
   \node at (0,0)[below left]{\footnotesize$\sqrt{-\lambda_1(0)}$};
   \fill(5*1.32,0)circle(2pt);
   \node at (5*1.32,0)[below right]{\footnotesize$\sqrt{-\lambda_2(0)}$};
    \draw[dotted] (5*0.9307,5*1.1474)--(0.05,5*1.1474);
    \node at (-0.05,5*1.1474)[left]{$|\tau_s|$};
    \draw[very thin] (-0.07, 5*1.1474) -- (0.07, 5*1.1474);
    \end{tikzpicture}
    \caption{The function $f(\mu):=\mu\sqrt{2\mu\cot(\mu)+1-\mu^2}$.}
    \label{fig:tau}
\end{figure}
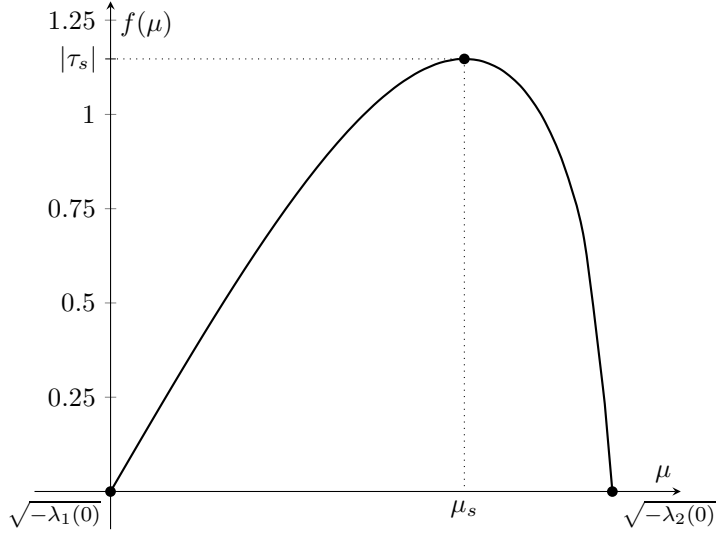

For $\lambda \in (\lambda_2(0)-\tau^*, 0]$ and $\mu \in [0,\sqrt{\lambda_2(0)}]$, the equality $f(\mu)=|\tau|$ implies \eqref{eq.mu} if $\tau \neq 0$. For $\tau=0$, the assertion follows form Corollary \ref{c.A.nopert}. Thus by Proposition \ref{prop.spect.prop} the value $\lambda=-\mu^2$ is an eigenvalue of $-\calA_{\tau}$, which satisfies $\lambda \in [\lambda_2(0), 0]$ and $\lambda \in (\lambda_2(0), 0)$ if $\tau\neq 0$.

On the other hand, let $\lambda\in (\lambda_2(0) -\tau^*, 0]$ be an eigenvalue of $-\calA_\tau$ and $\mu = \sqrt{-\lambda}$.
If $\tau=0$, the assertion follows from Corollary \ref{c.A.nopert}. 
Now let $|\tau|>0$. As $\spb(-\calA)< 0$, it is $\mu>0$. Moreover,
\[0 <\mu=\sqrt{-\lambda} < \sqrt{\tau^*-\lambda_2(0)} < \sqrt{8} < \pi.\] 
Since $\mu$ has to satisfy \eqref{eq.mu} we must have $2\mu\cot(\mu)+1-\mu^2=\frac{|\tau|^2}{\mu^2}> 0$. 
We note that the function $\mu \mapsto 2\mu\cot(\mu) +1-\mu^2$ is continuous on $(0,\pi)$ with a single zero at $\sqrt{-\lambda(0)}$, 
at which it changes sign from positive to negative. Thus, we must have $\mu < \sqrt{-\lambda(0)}$ as claimed.
\end{proof}

In a next step, we precisely characterize the value of $\tau$ for which we are in the critical case (B) of Lemma \ref{lem.spectral.projection}. 

\begin{prop}\label{prop.taus}
\label{p.cp}
    Let $\mu_s$ be such that $f(\mu_s)$ is maximal, i.e.\ $\mu_s$ is
    the unique solution of the equation 
    \begin{equation}\label{eq.maximum}
    1=2\mu^2-3\mu\cot\mu +\mu^2\csc^2\mu
    \end{equation}
    {\rm (}$\mu_s \approx 0.9307${\rm )} and 
    
   \begin{equation}\label{taus} 
   \tau_s\coloneqq f(\mu_s)=\sqrt{2\mu_s^3\cot\mu_s + \mu_s^2-\mu_s^4}\approx 1.1474.
   \end{equation} 
   Then the following hold true:
    \begin{enumerate}
    [\upshape (a)]
        \item  For $|\tau| \in (0, \tau_s)$, we are in case {\upshape(A)} of Lemma \ref{lem.spectral.projection}. Furthermore, $\lambda_2(0)<\lambda_2(\tau)<\lambda_1(\tau)=\spb(-\calA)<0$.
        \item For $|\tau|= \tau_s$,  we are in case {\upshape(B)} of Lemma \ref{lem.spectral.projection} and  $\lambda(\tau)=-\mu_s^2<0$.
        \item For $|\tau| \in (\tau_s, \tau^*)$, we are in case {\upshape(C)} of Lemma \ref{lem.spectral.projection}. 
    \end{enumerate}
\end{prop}
\begin{proof}
As $|\tau|<\tau^*$, it follows from Lemma \ref{lem.C} that $\lambda$ is a real eigenvalue of $-\calA_\tau$ if and only if
$\mu=\sqrt{-\lambda}$ solves $f(\mu)=\tau$.
Since $f$ is strictly concave, it has a unique maximum at which $f(\mu)=\tau$ has exactly one solution. 
This maximum can be found by setting $f'(\mu)=0$. Noting that $f(\mu)>0$ on $J$, we can equivalently solve
\[ 0=2f'(\mu)f(\mu) = (f^2)'(\mu) =
-2\mu(-1 + 2 \mu^2 - 3 \mu \cot(\mu) + \mu^2 \csc^2(\mu))=(f^2)'(\mu).
\] 
Thus the equation $f'(\mu)=0$ is equivalent to \eqref{eq.maximum} and we see that for $\mu=\mu_s$ we 
are in case (B) of Lemma \ref{lem.spectral.projection}. For $0<|\tau|<\tau_s$, there are exactly two real solutions of
$f(\mu) = |\tau|$ and we are in case (A) of Lemma \ref{lem.spectral.projection}. 
If $\tau^*>|\tau|>\tau_s$, there are no real eigenvalues, so we have to be in case (C). 
\end{proof}

In a final step, we investigate for which $|\tau| \in (0,\tau_s)$ it is possible to choose a strictly positive eigenfunction. 
It will turn out that this is only true for $|\tau|$ up to a slightly smaller threshold $\tau_p < \tau_s$.
\begin{prop}\label{prop.taup}
    Let $\mu_p$ be the smallest strictly positive solution of $\cot\mu = \mu$ {\upshape(}i.e  $\mu \approx 0.86033${\upshape)} and 
    \begin{equation}\label{taup}
        \tau_p:=f(\mu_p)=\sqrt{\mu_p^2+\mu_p^4}=\frac{\mu_p}{\sin(\mu_p)}\approx 1.1349.
    \end{equation} 
    Then $\tau_p<\tau_s$.
    Recall from Proposition \ref{p.cp} that for $|\tau|\in (0, \tau_s)$, we have $\spb(-\calA_\tau) \in \sigma(-\calA_\tau)$ and the
    corresponding eigenspace $\mathrm{Eig}(-\calA_\tau, \spb(-\calA_\tau)) = \mathrm{span}\{u_0\}$ is one-dimensional. Then the following is true:
    \begin{enumerate}[\upshape (a)]
        \item For $|\tau|\in (0, \tau_p)$, we can choose $u_0$ strictly positive on $[0,1]$.
        \item For $|\tau| = \tau_p$, we can choose $u_0$ positive on $[0,1]$ and strictly positive on $(0,1)$, but $u_0(x) = 0$ for some $x\in \Gamma$.
        \item For $|\tau| \in (\tau_p, \tau_s)$, $u_0$ changes sign.     
        \end{enumerate}
\end{prop}

\begin{proof} 
Let $\mu_p$ be the smallest positive solution of $\cot(\mu)=\mu$. Then, approximately,  $\mu_p \approx 0.86033<\mu_s$. 
Moreover, if $|\tau_p|=f(\mu_p)$, then  $-\mu_p^2 = \lambda_1(\tau_p) = \spb(-\calA_{\tau_p})$.
Note that
\begin{align*}
    f(\mu_p) &=\sqrt{2\mu_p^3\cot\mu_p + \mu_p^2-\mu_p^4}=\sqrt{\mu_p^4+\mu_p^2}\\
    &=\sqrt{\mu_p^2(1+\cot^2(\mu_p))}=\frac{\mu_p}{\sin(\mu_p)} \approx 1.13491<\tau_s.
\end{align*} 

It turns out that (unlike the eigenvalues) the eigenfunctions do depend on the sign of $\tau$, so we make a case distinction. For our means it suffices to calculate the eigenfunction corresponding to the spectral bound, which corresponds to the solution of 
$f(\mu)=|\tau|$ in the range $0\leq \mu \leq \mu_s$ (cf. Figure \ref{fig:tau}).

\smallskip

\emph{Case 1: $\tau>0$}.

Set \[v^+(x)=\cos(\mu x)-\frac{\tau \cos(\mu)+\mu^2}{\mu+\tau \sin(\mu)}\sin(\mu x).\] One can check that $v^+(x)$ is an eigenfunction to the eigenvalue $-\mu^2$ of the operator $-\calA_\tau$ for $\tau=f(\mu)$. 
Note that $\mu+\tau\sin(\mu)>0$ for $\mu \in J$. 

At $\tau=\tau_p$ we have $\tau_p \sin(\mu_p)=\mu_p$ and $\cos(\mu_p)=\mu_p \sin(\mu_p)$, thus the eigenfunction to $\spb(-\calA_{\tau_p})=-\mu_p^2$ is given by  \[v^+(x)=\cos(\mu_p x)-\mu_p\sin(\mu_p x),\] which is a strictly positive function on $[0,1)$ with a zero in $x=1$. This proves (b).

More generally for any $\mu \in J$ the function $v^+$ has a zero if and only if $\tan(\mu x)=\frac{\mu+f(\mu)\sin(\mu)}{\mu^2+f(\mu)\cos(\mu)}$.
As $x\mapsto \tan (\mu x)$ is a strictly increasing function that maps $[0, \frac{\pi}{2\mu})$ onto $[0,\infty)$, for fixed
$\mu$ this equation has a unique solution $x_\mu$ in $(0, \frac{\pi}{2\mu})$. We note that $x_\mu \in (0,1)$ if and only if $\tan(\mu)> \frac{\mu + f(\mu)\sin(\mu)}{\mu^2+f(\mu)\cos(\mu)}$; and the latter can be shown to be equivalent to $\mu> \mu_p$. This proves (a) and (c).\smallskip

\emph{Case 2: $\tau<0$}.

For negative $\tau$, we observe that
\[v^-(x)=\sin(\mu x)-\frac{\mu^2\sin(\mu)-\mu \cos \mu}{\mu\sin(\mu)+\mu^2\cos(\mu)-\tau}\cos(\mu x)\] 
is an eigenfunction of $-\calA_\tau$ for the eigenvalue $-\mu^2$, where $\tau=-f(\mu)$. Note that the denominator is strictly positive for $\mu \in J$.

The question of a sign change of the first eigenfunction reduces to whether 
\begin{equation}
    \label{eq.help}
    \tan(\mu x)=\frac{\mu^2 \sin(\mu)-\mu \cos(\mu)}{\mu \sin(\mu)+\mu^2 \cos(\mu)+f(\mu)},
    \end{equation} 
occurs for $x \in [0,1]$.
For $0<\mu<\mu_p$, the right hand side of \eqref{eq.help} is negative, so there is no equality in \eqref{eq.help} for $x\in [0,1]$.
For $\mu = \mu_p$, it is $v^-(x) = \sin(\mu_p x)$, which has a zero in $x=0$. Thus, we have proved (a) and (b). To prove (c),
note that if $\mu>\mu_p$, the right hand side of \eqref{eq.help} is strictly positive and 
\[ 
\tan(\mu)>\frac{\mu^2 \sin(\mu)-\mu \cos(\mu)}{\mu \sin(\mu)+\mu^2 \cos(\mu)+f(\mu)}.
\] 
By continuity, \eqref{eq.help} has a solution $x\in (0,1)$.
\end{proof}

Now we can prove the main result of this section.
\begin{proof}[Proof of Theorem \ref{t.ev.pos}]
(a). If $\calT$ is positive, then $\tau B_{22}$ satisfies the positive minimum principle by Theorem \ref{t.characterization}. However, 
for $\binom{1}{0}, \binom{0}{1} \in \R_+^2$, we have $\langle \binom{1}{0}, \binom{0}{1}\rangle = 0$, 
$\langle \tau B_{22} \binom{1}{0},\binom{0}{1}\rangle=-\tau$, and $\langle \tau B_{22} \binom{0}{1},\binom{1}{0}\rangle=\tau$. 
Both values are positive only if $\tau=0$. This proves necessity of $\tau=0$. That $\calT$ is positive for $\tau=0$ follows immediately
from Theorem \ref{t.ev.pos}.

To prove (b), first observe that Theorem \ref{t.hoelder} implies the smoothing condition \eqref{eq.smoothing-Eu}. 
 Propositions \ref{p.spectrum}(a), \ref{prop.taus}(a), and \ref{prop.taup}(a) show that condition (iii) from Theorem \ref{t.eventual-pos-characterize} is satisfied for the operator $-\calA_\tau$. 
As the respective eigenfunctions are continuous and have no zeros, they can be chosen to satisfy $v \geq \delta\one$. 
The conditions on the dual semigroup follow from the fact that $-\calA_\tau^*=-\calA_{-\tau}$ 
and that we can choose $\psi$ as the strictly positive element in $\ker(\spb(-\calA_\tau)I+\calA_{-\tau})$.

(c) can be deduced by showing that in all sub-cases at least one of the conditions from Theorem \ref{t.eventual-pos-characterize}(iii) is violated. In the sub-cases (i) and (ii) we use Proposition \ref{prop.taus}(a) and Proposition \ref{prop.taup}(b) and (c). For (iii) and (iv) we use the sub-cases (b) and (c) of Proposition \ref{prop.taus}, respectively. 
\end{proof}

\appendix

\section{Bounded perturbations of \texorpdfstring{weak$^*$-}{weak star }semigroups}

Throughout this appendix, let $M$ be a compact, separable metric space
and $\mu$ be a finite Borel measure on $M$ such that $\mu(B(x, \eps)) >0$ for every $x\in M$ and $\eps>0$.
We are interested
in the space $L^\infty(M, \mu)$. We start with a characterization
of adjoint operators.

\begin{lem}
    \label{l.adjoint}
    Let $T\in \calL(L^\infty(M, \mu))$. Then $T$ is an adjoint operator, i.e.\ there is some $\tilde T\in \calL(L^1(M, \mu))$
    with $\tilde T^*=T$, if and only if whenever $(f_n)\subset L^\infty(M, \mu)$ is a bounded sequence with $f_n \to f$ pointwise almost
    everywhere, it follows that $Tf_n \weakstar Tf$.
\end{lem}

\begin{proof}
    If $T=\tilde T^*$ and $(f_n)$ is a uniformly bounded sequence that
    converges to $f$ almost everywhere, then for $g\in L^1(M, \mu)$
    we have
    \[
    \langle g, Tf_n\rangle = \langle \tilde Tg, f_n\rangle \to \langle \tilde T g, f\rangle = \langle g, Tf\rangle
    \]
    by dominated convergence.\smallskip

   Conversely, assume that $T$ satisfies the stated continuity condition. Let $g\in L^1(M, \mu)$. We claim that $T^*g \in L^1(M, \mu)$ where we identify $L^1(M, \mu)$ canonically with a closed
    subspace of $L^\infty(M, \mu)^*$. To see this, put $\nu(A) = \langle T^*g, \one_A\rangle$. If
    $(A_n)_{n\in \N}$ is a sequence of pairwise disjoint Borel sets,
    we define $f_n \coloneq \one_{\bigcup_{k=1}^n A_n}$ and $f \coloneqq \one_{\bigcup_{k=1}^\infty A_k}$. Then the sequence
    $f_n$ is uniformly bounded and converges to $f$ almost everywhere.
    By assumption,
    \[
    \nu\Big(\bigcup_{k=1}^\infty A_k\Big) = \langle g, Tf\rangle = 
    \lim_{n\to \infty}\langle g, Tf_n\rangle = \lim_{n\to \infty}
    \sum_{k=1}^n\nu(A_k).
    \]
    This proves that $\nu$ is a Borel measure. As clearly $\nu\ll\mu$,
    there is a function $h\in L^1(M, \mu)$ with $\dd\nu = h\dd\mu$.
    This proves $T^*g=h\in L^1(M, \mu)$.

    Setting $\tilde T\coloneqq T^*|_{L^1(M, \mu)}$, it follows that
    $\Tilde T^*=T$.
\end{proof}

\begin{defn}
    \label{def.ws}
    A \emph{weak$^*$}-semigroup on $L^\infty(M, \mu)$
    is a family $(T(t))_{t\geq 0}$ of bounded linear operators on $(L^\infty(M,\mu))$ such that
    \begin{enumerate}[label=\upshape (\roman*)]
        \item $T(0) = I$ and $T(t+s) = T(t)T(s)$ for all $t,s\geq 0$;
        \item every operator $T(t)$ is an adjoint operator;
        \item for every $f\in L^\infty(M, \mu)$ the orbit $t\mapsto T(t)f$
        is weak$^*$-continuous.
    \end{enumerate}
    The \emph{weak$^*$}-generator $A$ of $(T(t))_{t\geq 0}$ is defined by
    \[
    Af \coloneqq \mathrm{weak}^*-\lim_{t\to 0}\frac{1}{h}(T(h)f-f),
    \]
    on the domain $D(A)$, consisting of all $f$ for which this limit exists.
\end{defn}

We show next that a weak$^*$-semigroup is just the adjoint of a strongly continuous semigroup
on $L^1(M,\mu)$. For more information on adjoint semigroups, we refer to \cite{vN92}.

\begin{lem}
    \label{l.adjoint-semigroup}
    Let $(T(t)_{t\geq 0}$ be a weak$^*$-semigroup. Then there exists a strongly continuous
    semigroup $(\tilde T(t))_{t\geq 0}$ on $L^1(M, \mu)$ such that $\tilde T(t)^* = T(t)$ for 
    all $t\geq 0$, i.e.\ $(T(t))_{t\geq 0}$ is an \emph{adjoint semigroup}.
    If $\tilde A$ denotes the generator of $\tilde T$, then $A =\tilde A^*$ is the weak$^*$-generator of $T$.
\end{lem}

\begin{proof}
    If $\tilde T(t)$ is such that $\tilde T(t)^*=T(t)$, then $(\tilde T(t))_{t\geq 0}$ clearly satisfies the semigroup law. 
    Moreover, as the orbits of $T$ are weak$^*$-continuous, the orbits of $\tilde T$ are weakly continuous. 
    By \cite[Theorem I.5.8]{en00}, $\tilde T$ is strongly continuous. The statement about the weak$^*$-generator
    follows from \cite[\S II.2.5]{en00}.
\end{proof}

\begin{prop}
    \label{p.boundedperturbation}
    Let $(T(t))_{t\geq 0}$ be a weak$^*$-semigroup with weak$^*$-generator $A$ and let
    $B\in \calL(L^\infty(M, \mu))$ be an adjoint operator. Then $A+B$ is the weak$^*$-generator
    of a weak$^*$-semigroup $(S(t))_{t\geq 0}$. Moreover, it holds that
    \begin{equation}
    \label{eq.duhamel}
    S(t)f = T(t)f +\int_0^t T(t-s)BS(s)f\, ds
    \end{equation}
    for all $f\in L^\infty(M,\mu)$ and $t\geq 0$. Here, the integral in \eqref{eq.duhamel}
    has to be understood as a weak$^*$-integral. Finally, if $T$ is contractive and $B$ is dissipative
    then also $S$ is contractive.
\end{prop}

\begin{proof}
    Let $\tilde B\in \calL(L^1(M, \mu))$ be such that $\tilde B^*=B$. Moreover, let $(\tilde T(t))_{t\geq 0}$ be the strongly
    continuous semigroup on $L^1(M, \mu)$ such that $\tilde T(t)^*=T(t)$ and $\tilde A$ be the generator of $\tilde T$, see
    Lemma \ref{l.adjoint}. By \cite[Theorem III.1.3]{en00}, $\tilde A + \tilde B$ is the generator of a strongly continuous semigroup
    $(\tilde S(t))_{t\geq 0}$. The Duhamel formula \eqref{eq.duhamel} for $\tilde T$ and $\tilde S$ follows from \cite[Corollary III.1.7]{en00}.
    Taking adjoints, the claim follows. For the last statement first observe that if $B$ is a bounded, dissipative operator,
    then it is m-dissipative as some point on the positive real axis belongs to the resolvent set of $B$, see \cite[Proposition II.3.14]{en00}.
    But then it follows that its pre-adjoint $\tilde{B}$ is also m-dissipative. If $T$ is contrative, then so is $\tilde T$, and it follows
    from \cite[Proposition III.2.7]{en00} that the semigroup generated by $A+B$ is contractive.
\end{proof}

\begin{defn}
    \label{def.sf}
    Let $(T(t))_{t\geq 0}$ be a weak$^*$-semigroup and $X$ be a closed subspace of $C(M)$. 
    Then $T$ is called \emph{strong Feller semigroup with respect to $X$} if
    \begin{enumerate}[label=\upshape (\roman*)]
        \item $T(t)f\in X$ for every $f\in L^\infty(M, \mu)$ and $t>0$;
        \item for every $f\in X$ it holds that $T(t)f \to f$ with respect to $\|\cdot\|_\infty$ as $t\to 0$. 
    \end{enumerate}
\end{defn}

\begin{rem}
    Usually, a strong Feller operator is defined as a kernel operator on $B_b(M)$, the space of all bounded, measurable
    functions on $M$, that maps $B_b(M)$ to $C(M)$. However, if $q : B_b(M) \to L^\infty(M, \mu)$, denotes the quotient map that
    maps a bounded measurable function to its equivalence class modulo equality $\mu$-almost everywhere, 
    and if $T\in \calL(L^\infty(M, \mu))$ is an adjoint operator that takes values in $C(M)$, 
    then one can show that $T\circ q$ is a strong Feller operator in the classical sense. 
    In particular, the fact that $T$ is an adjoint operator implies that $T\circ q$ is a kernel operator.
    For more information we refer to \cite[Section 4.1]{dhk24}. We would also like to point out that if $T\in \calL(L^\infty(M, \mu))$
    is an adjoint operator taking values in $C(M)$, then $T_0\coloneqq T|_{C(M)}$ is weakly 
    compact and $T_0^2$ is compact, see \cite[Theorem 4.4]{dhk24}.
\end{rem}

The following Theorem could be obtained as a special case of a perturbation theorem for more general strong Feller semigroups,
see \cite[Theorem 3.3]{kunze13} (see in particular \cite[Example 3.4]{kunze13}) or \cite[Theorem 3.2]{kk22}. However, in our situation,
where we consider weak$^*$-semigroups, we can give an easier and direct proof.

\begin{thm}
    \label{t.perturb}
    Let $T=(T(t))_{t\geq 0}$ be a weak$^*$-semigroup with weak$^*$-generator $A$ and $B\in \calL(L^\infty(M, \mu))$ be an
    adjoint operator. Moreover, let $S=(S(t))_{t\geq 0}$ be the weak$^*$-semigroup generated by $A+B$. If $X$ is a closed subspace
    of $C(M)$ and $T$ is
    a strong Feller semigroup with respect to $X$, then so is $S$.
\end{thm}

\begin{proof}
    Let us first prove that $S(t)f \in C(M)$ for every $f\in L^\infty(M, \mu)$ and $t>0$. To that end, fix $t>0$ and note
    that $T(t-s)BS(s)f\in C(M)$ for $0<s<t$, since $T(t-s) L^\infty(M) \subset C(M)$. For fixed $x\in M$, let $g_n = \mu (B(x, n^{-1}))^{-1}\one_{B(x,n^{-1})}$. It follows that
    \[
    \langle g_n, T(t-s)BS(s)f\rangle \to [T(t-s)BS(s)f](x)\quad\mbox{ as } n\to \infty.
    \]
    This implies that for fixed $x$ the map $s\mapsto [T(t-s)BS(s)f](x)$ is measurable. We may thus define the function
    $h$ on $M$ by setting
    \begin{equation}
        \label{eq.pointwise}
        h(x) \coloneqq \int_0^t[T(t-s)BS(s)f](x)\dd s.
    \end{equation}
    
    Then $h\in C(M)$. Indeed, if $x_n \to x$, it follows that $[T(t-s)BS(s)f](x_n) \to [T(t-s)BS(s)f](x)$ as $n\to \infty$ and, as
    $\sup_{s\in [0,t]}\|T(t-s)BS(s)\|<\infty$, continuity of $h$ follows from the dominated convergence theorem. To see
    that actually $h\in X$, we first note that  integrating \eqref{eq.pointwise} with respect to a Borel measure $\nu$ yields
    \[
    \langle h, \nu\rangle = \int_0^t\langle T(t-s)BS(s)f, \nu\rangle \dd s.
    \]
    As $T$ is a strong Feller semigroup with respect to $X$, it follows that for every $s\in (0,t)$ the function $T(t-s)BS(s)f$
    is an element of $X$. If $h\not\in X$, the Hahn--Banach theorem implies that there exists a measure $\nu$ with 
    $\langle g,\nu\rangle=0$ for all $g\in X$ but $\langle h, \nu\rangle \neq 0$. This is a contradiction. At this point, it follows from the Duhamel formula \eqref{eq.duhamel} that $S(t)f\in X$.\smallskip

    Next note that $C\coloneqq \sup_{t\in (0,1]}\sup_{s\in [0,t]}\|T(t-s)BS(s)\| <\infty$. It follows that
    \[
    \Big\|\int_0^t T(t-s)BS(s)f\dd s\Big\| \leq \int_0^t C\|f\|\dd s\leq Ct\|f\| \to 0\quad \mbox{ as } t\to 0
    \]
    for every $f\in C(M)$. As $T(t)f \to f$ for such $f$, condition (ii) in the definition of strong Feller semigroup
    follows once again from \eqref{eq.duhamel}.
\end{proof}

\newcommand{\etalchar}[1]{$^{#1}$}

\end{document}